\documentclass{dcds-bOF}
\usepackage{amsmath}
\usepackage{paralist}
\usepackage{graphics}
\usepackage{epsfig}
\usepackage{graphicx}
\usepackage{epstopdf}
\usepackage{auto-pst-pdf}
\usepackage[colorlinks=true,breaklinks]{hyperref}
\hypersetup{urlcolor=blue, citecolor=red}

\def\currentvolume{X}
 \def\currentissue{X}
  \def\currentyear{200X}
   \def\currentmonth{XX}
    \def\ppages{X--XX}
     \def\DOI{10.3934/dcdsb.2021281}

\numberwithin{equation}{section}

\newtheorem{theorem}{Theorem}[section]

\newtheorem{lemma}[theorem]{Lemma}

\theoremstyle{definition}

\newtheorem{remark}{Remark}
\newtheorem{example}{Example}
\newtheorem{assumption}{Assumption}

\title[Multistability and cycles]{Stabilizing multiple equilibria and cycles with noisy prediction-based control}
\author[Elena Braverman and Alexandra Rodkina]
{}

\email{maelena@ucalgary.ca}
\email{alexandra.rodkina@uwimona.edu.jm}

\thanks{
E. Braverman is a corresponding author.
The first author is supported by NSERC grant
RGPIN-2020-03934. }

\begin{document}
\maketitle

\centerline{\scshape Elena Braverman}
\medskip
{\footnotesize
 \centerline{Dept. of Math. and Stats., University of Calgary}
 \centerline{2500 University Drive N.W. Calgary, AB,  T2N 1N4, Canada}}

\medskip

\centerline{\scshape  Alexandra Rodkina}
\medskip
{\footnotesize
 \centerline{Department of Mathematics,
the University of the West Indies}
 \centerline{Mona Campus, Kingston, Jamaica}}


\bigskip

\centerline{(Communicated by Martin Rasmussen)}

\begin{abstract}
Pulse stabilization of cycles with Prediction-Based Control including noise and
stochastic stabilization of maps with multiple equilibrium
points is analyzed for continuous but, generally, non-smooth maps.
Sufficient conditions of global stabilization are obtained.
Introduction of noise can relax restrictions on the control intensity.
We estimate how the control can be decreased with noise and verify it numerically.


\end{abstract}

\section{Introduction}

Consider a difference equation
\begin{equation}
\label{intro0}
x_{n+1}=g(x_n), \,\, n\in \mathbb N, \,\, x_0\in \mathbb R,
\end{equation}
with the continuous function $g:\mathbb R\to \mathbb R$  having several fixed points in the set $\mathbb K:=\{K_j, \,  j=0, 1,2, \dots, j_0\}$, $K_{j_0}\le \infty$.
We aim to stabilize simultaneously  all equilibria from $\mathbb K$ with odd indexes applying Prediction-Based Control (PBC) method 
\cite{uy99} with variable or stochastically  perturbed control  $\alpha_n ~\in~(0,1)$
\begin{equation}
\label{eq:PBCvar}
x_{n+1}=g(x_n)-\alpha_{n+1}(g(x_n)-x_n), \quad x_0\in \mathbb R,
\quad n\in {\mathbb N}_0 := \mathbb N\cup \{0\}.
\end{equation}
If $g$ is a unimodal function with a negative Schwarzian derivative, $\alpha_n \equiv \alpha$ are constant,
global and local stability of the unique positive equilibrium coincide for original map \eqref{intro0}, see \cite{Singer}, and controlled equation \eqref{eq:PBCvar}, see \cite{FL2010}.
This follows from the fact that the controlled map inherits the same properties \cite{FL2010}.
More sophisticated behaviour of maps with PBC is observed if either $g$ has multiple critical and equilibrium points, or
\eqref{intro0} is considered with control \eqref{eq:PBCvar} at every $k$-th step only. This corresponds to pulse control which can be applied in both deterministic and stochastic cases \cite{BKR2020,LizPBC2012,LP2014}. Pulse control can be viewed as PBC for the iterated map $g^k$. Even for unimodal $g$, after applying PBC, all the values at critical points become different which does not allow to apply \cite{Singer} for the pulse control.  Stabilization of equilibrium points of iterated maps with PBC corresponds to stabilizing either an equilibrium or a cycle of the original map.

In many practical applications, in particular for models of population dynamics, one of two one-side
Lipschitz constants for a stable equilibrium can be less than one, while the other can be quite large. This motivates us to concentrate on one-side constants.

For an arbitrary number $j_0+1$ of equilibrium points in $\mathbb K$,  with $g(x)-x$ changing sign at every point, starting with plus, we assume that at each point $K_{2j+1}$, the  function $g$ satisfies a one-side Lipschitz condition:
  \begin{equation}
  \begin{split}
  \label {cond:intr}
      &K_{2i+1}-g(x)\leq  L^+_{2i+1} (x-K_{2i+1}), \quad x\in (K_{2i+1}, K_{2i+2}),   \\
    &g(x)-K_{2i+1} \leq  L^-_{2i+1} (K_{2i+1}-x), \quad  x\in (K_{2i}, K_{2i+1}),
    \end{split}
		\quad 2i+2 \leq j_0,
    \end{equation}
   where one of $ L^-_{2i+1}$ and $ L^+_{2i+1}$ can be infinite. The expression $g(x)-x$ is supposed to be either non-negative or non-positive on each $(K_i, K_{i+1})$, and    may have other  equilibrium points inside  $(K_i, K_{i+1})$.
The function $g$ does not necessarily map each $(K_i, K_{i+1})$ and even  all $(K_0, K_{j_0})$ to itself.

The low threshold $L/(L+1)$ of the control  is calculated, based on the minimum $L$ of the  left and the right Lipschitz constants at each $K_{2j+1}$, which guarantees that, once a solution is in this smoother interval with a less steep $g$, it stays there.
However, this value might not be enough to prevent a solution from overshooting, getting out of $(K_0, K_{j_0})$ and from switching between different intervals, attending some of them an infinite number of times.
Due to the sign restrictions on $g(x)-x$, a solution either converges to one of equilibrium points or circulates infinitely between intervals.
The main goal of this paper is to find the least lower bound of the control parameter $\alpha_n$ which prevents this infinite circulation.

Distinction between points $K_{2j+1}$, based on the sizes of the left or the right Lipschitz constants, allows us  to split $(K_0, K_{j_0})$ into the union of  blocks of non-overlapping intervals.   We show that the infinite fluctuation of a solution is possible only inside some of those blocks, while inside the others  there is no circulation at all, after the application of the first stage of control. For each block with possible fluctuation, we  find the least low bound for the control preventing circulation.  We prove that if the control does not exceed this bound for some block, there is a two-cycle inside the block. The main deterministic  result states that  when \eqref{cond:intr} holds, there exists a  control such that  a solution  $x$ to \eqref{eq:PBCvar} with $x_0\in (K_0, K_{j_0})$,  converges  either to an odd-numbered equilibrium or to an equilibrium belonging to $(K_p, K_{p+1})\subset (K_0, K_{j_0})$.

Fig.~\ref{fig:Ricker2} illustrates the second iterate $f^2$ of the Ricker map $f(x)=xe^{r(1-x)}$ with $r=2.7$.  As it is seen from Fig.~\ref{fig:Ricker2}, $f^2$  has quite a big Lipschitz constant  on the interval $(0, K_1)$, so the stability PBC control parameter for it can be found based on the right-side constant.
However, since $f^2$ is continuously differentiable,  on some interval to the left of $K_1$, the  Lipschitz constant is close to $(f^2)'(K_1)\approx L$. Inspired by this example,  we consider a function $g$ having four fixed points, which allows to decrease the threshold $L/(L+1)$ of the control  to the value $(L-1)/(L+1)$. We prove deterministic, as well as stochastic version of this result.

\begin{figure}[ht]
\centering
\includegraphics[height=.22\textheight]{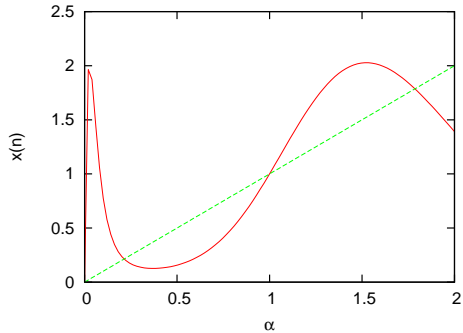}
\caption{The second iteration of the Ricker map for $r=2.7$.}
\label{fig:Ricker2}
\end{figure}

When $\alpha_n$  is random, it has the form $\alpha_n=\alpha+\ell\xi_{n}$, i.e.  a deterministic  control $\alpha$ is perturbed by a noise $\xi$ with intensity $\ell$. We assume that  random variables $\xi_n$, $n\in \mathbb N$, are independent, identically distributed and $|\xi_n|\le 1$. We also suppose that the value of the noise $\xi_n$  can be however close to one with a positive probability, which allows  to apply  the Borel-Cantelli lemma  and conclude that, for any   however big $\mathbf J\in \mathbb N$,
the value of the noise $\xi_n$  can stay however close to 1 for  $\mathbf J$ consecutive number of steps keeping $\alpha+\ell\xi_i$ close to $\alpha+\ell$.  This event will happen after some random moment $\mathcal N$. We show that, once $\alpha+l\xi_n>\underline \beta$, where $\underline \beta \in (0,1)$ is a prescribed constant, for $n = \mathcal N, \dots,  \mathcal N+\mathbf J$  a solution is driven into the area from where  circulation between intervals becomes impossible. This leads to convergence of a solution to one of the equilibrium points.  Thus the aim of the introduction of stochastic control is two-fold: to demonstrate the range of noise which keeps stability for the same interval of $\alpha_n$, and to improve deterministic results in a sense that there is stability for $\ell>0$, while there is no stability in deterministic case $\ell=0$ (compare with \cite{BR2019,Medv}).
A sharp result on the minimal value of the control is aligned with the idea of excluding a possible two-cycle \cite{Cull}.

The results of the paper are novel even in a deterministic setting, for instance,  compared to the multistability case considered in \cite[Theorem 2.5]{LizPBC2012}. The advantage of deterministic results in the present paper is that one-side local Lipschitz constants are taken into account, and even in the case of one-side infinite derivatives, stabilization with PBC can be achieved.  In addition, a solution can switch between different intervals $(K_p, K_{p+1})$ a finite number of times. On the other hand, in  \cite{LizPBC2012}, where existence of a two-side global Lipschitz constant was assumed, stabilization was achieved when the  lower bound for $\alpha$ is $1-1/L$, which is smaller than $L/(1+L)$. However, in the case when the derivative at the equilibrium points, where the sign $g(x)-x$ switches from positive to negative, is finite, we extend  the result from
\cite{LizPBC2012} and, moreover, improve it in the stochastic case.

Note that as $\alpha_n$ approaches one, the controlled map in \eqref{eq:PBCvar} becomes closer to the identity map and, once one-side Lipschitz constants are finite, stabilization is possible for $\alpha_n$ close enough to one, but separated from one (otherwise, we can stabilize non-fixed points of the original map). However, this is no longer true if one of one-side Lipschitz constants is infinite. In addition, whenever one of these constants is quite large, which is a typical situation for iterates of chaotic maps, choosing control intensive enough to map each union of adjacent segments surrounding a point that is potentially stable after a control application, can lead to significant overshoot in estimating the required parameter. The optimal way to find the required control bound is to trace possible two-cycles and choose a minimal bound excluding the existence of cycles.

We present three examples  which illustrate our theory, in deterministic as well in stochastic settings. One of these result has  one-side infinite derivatives, another two are  iterates of the Ricker map with $r=2.7$ and $r=3.5$. Bifurcation diagrams for these maps demonstrate how the appropriate noise level $\ell$ provides stability for $\alpha$ which does not work in pure deterministic framework.

The article has the following structure.  In Section~\ref{sec:4eq},  we consider the case of $g$ with four equilibrium points to illustrate the main ideas,
under an assumption which combines condition~\eqref{cond:intr}  and the case of differentiable at $K_1$ and $K_3$ function $g$.
In Sections~\ref{subsec:aux4eq}-\ref{subsec:trap2} we formulate auxiliary statements
and establish both the methods and the results for deterministic PBC with a variable control parameter.
  In Section~\ref{subsec:stoch} we explore the case when the control is perturbed by an additive noise. While the main results are stated in the text, all the  proofs are deferred to the Appendix.
 In Section~\ref {sec:arbeq}  we briefly discuss the case of an arbitrary number of equilibrium points. Most details including classification of intervals, auxiliary statements, proofs in deterministic and stochastic cases are in the Appendix, in particular, in Sections~\ref{sec:6_1}-\ref{subsec:arbap}.
Section \ref{sec:ex} illustrates our results with numerical examples.
Discussion of the results and future research directions in Section~\ref{sec:disc} concludes the text of the paper.

\section{The function with four equilibrium points}
\label{sec:4eq}

The model with four equilibrium points was inspired by the second iteration of population dynamics models. As an example, we take a Ricker map for $r>2$ when a two-cycle exists. Fig.~\ref{fig:Ricker2} illustrates that, for example, the left-side Lipschitz constant at $K_1$ significantly exceeds the right constant and also the derivative at the point, leading to greater required control values.
We recall that this model corresponds to either an equilibrium or a two-cycle stabilization with
a pulse PBC control applied at every second step.

\subsection{Auxiliary  statements}
\label{subsec:aux4eq}

\begin{assumption}
\label{as:4eq1}
Let $g:[K_0, K_4)\to [K_0, K_4)$, where $K_4 \leq \infty$, be a continuous function such that
\begin{equation*}
\begin{split}
& g(K_j)=K_j,~ K_j<K_{j+1},  ~ g(x)-x> 0,~  x\in (K_0,K_1), ~ g(x)>K_0, ~  x\in (K_0, K_4),\\
&  g(x)-x \mbox{ alternates its sign on adjacent }   (K_j,K_{j+1}),  \,\,  j=0, 1,2, 3,
\\ & \mbox{ and for some} \,\,L>1, ~ \delta\in \left[0, \min_{0 \leq i \leq 3}\{K_{i+1}-K_i\}\right),
\\
& |g(x)-K_1|\le L|K_1-x|, ~~ x\in (K_1-\delta, K_2), \\ & |K_3-g(x)|\le L|x-K_3|, ~x\in (K_2, K_3+\delta).
\end{split}
\end{equation*}
\end{assumption}

\begin{lemma}
\label{lem:conv}
Let $g: \mathbb R \to \mathbb R$ be a continuous function, $x_n$ be a solution to \eqref{eq:PBCvar} with some $x_0\in \mathbb R$,  and $\alpha_n\in [0, a]$ for some $a\in (0, 1)$ and all $n\in \mathbb N$. If~ $\lim_{n\to \infty}x_n=x^*$ then $g(x^*)=x^*$.
\end{lemma}

Denote
\begin{equation}
\label{def:G}
G(v, x):=(1-v)g(x)+v x, \quad x\in \mathbb R, \, v\in [0, 1].
\end{equation}

\begin{lemma}
\label{lem:Gv}
Let  $g$ satisfy Assumption~\ref{as:4eq1}, and $G$ be defined as in \eqref {def:G}. Then,
\begin{enumerate}
\item [(i)] $G(1,x)=x$, $G(0,x)=g(x)$, $x\in [0, \infty)$, $G:[0,1]\times \mathbb R\to  \mathbb R$ is a continuous function;
\item [(ii)] $G(v,x)-x\to 0$ as $v\to 1$, and  $G(v,x)-g(x)\to 0$ as $v\to 0$,  uniformly on $[K_0, b]$ for each $b>K_0$;
\item [(iii)] for $1>a>b>0$, we have $g(x)>G(b, x)>G(a, x)>x$ if $g(x)>x$, while
$g(x)<G(b, x)<G(a, x)<x$, if $g(x)<x$;
\item [(iv)] for any $\mu_0 \in (0, 1)$, $\mu\in(\mu_0, 1)$ and
$\hat \mu:=\frac{\mu-\mu_0}{1-\mu_0}$, we have $G(\mu, x)=(1-\hat \mu)G(\mu_0, x)+\hat \mu x$.
\end{enumerate}
\end{lemma}

\subsection{The first trap}
\label{subsec:trap1}
We start with finding a control which forces the interval $[K_1-\delta, K_3+\delta]$ to become invariant (a trap) in a sense that as soon as $x_n \in [K_1-\delta, K_3+\delta]$ for some $n\in \mathbb N$, it stays there forever and converges either to $K_1$ or $K_3$.

\begin{lemma}
\label{lem:LL+2}
Let Assumption~\ref{as:4eq1} hold,  and $\alpha>L/(L+1)$. Let $\bar \alpha\in [\alpha, 1)$ and $\alpha_n\in [\alpha, \bar \alpha]$ for each $n\in \mathbb N$.  Then,
\begin{enumerate}
\item [(i)]
$G(\alpha, \cdot):[K_1, K_3]\to [K_1, K_3]$;

\item [(ii)]  if $\delta>0$, then, in addition to (i), $G(\alpha, \cdot):[K_1-\delta, K_1]\to [K_1-\delta, K_1]$ and  $G(\alpha, \cdot):[K_3, K_3+\delta]\to [K_3, K_3+\delta]$;

\item [(iii)]  for each $x_0\in [K_1-\delta, K_2)\cup (K_2, K_3+\delta]$ the solution to \eqref{eq:PBCvar} converges either to $K_1$ or to $K_3$.
\end{enumerate}
\end{lemma}

Now we improve the result of Lemma~\ref{lem:LL+2} by introducing  a control $ \underline \alpha_0$, which might be less than $L/(L+1)$,
\begin{equation}
\label{def:alpha0}
   \underline \alpha_0=\left\{\begin{array}{l}
   \frac L{L+1}, \,\, \mbox{if} \,\, \delta=0,\\
   \inf\biggl\{\beta\in \left(\frac {L-1}{L+1},1\right):\min_{x\in [K_1-\delta, K_2]}G(\beta, x)>K_1-\delta,\\
\max_{x\in [K_2, K_3+\delta]}G(\beta, x)<K_3+\delta\biggr\},\,\, \mbox{if} \,\, \delta>0.
   \end{array}\right.
 \end{equation}
 \begin{lemma}
  \label{lem:alpha0}
Let Assumption \ref{as:4eq1} hold, $\delta>0$, $\underline \alpha_0$ be defined as in \eqref{def:alpha0}.
Then,
\begin{enumerate}
\item [(i)] the set   introduced in \eqref{def:alpha0} for $\delta>0$ is non-empty, so $\underline \alpha_0$ is well defined;
\item [(ii)] $ \underline \alpha_0\le L/(L+1)$.
\end{enumerate}
\end{lemma}
 \begin{lemma}
\label{lem:delta01}
Let Assumption~\ref{as:4eq1} hold, $\underline \alpha_0$ be defined as in \eqref{def:alpha0}, $\alpha_*\in (\underline \alpha_0, 1)$, $\alpha^*\in (\alpha_*, 1)$ and $\alpha_n\in [\alpha_*, \alpha^*]$, for each $n\in \mathbb N$. Then,
\begin{enumerate}
\item [(i)]  $G(\alpha, \cdot):(K_{1}-\delta, K_3+\delta)\to (K_{1}-\delta, K_3+\delta)$, for each $\alpha>\underline \alpha_0$;
\item [(ii)] for $x_0\in (K_1-\delta, K_2)$ a solution to equation \eqref{eq:PBCvar} converges to $K_1$, while  for $x_0\in (K_2, K_3+\delta)$ it converges to $K_3$.
\end{enumerate}
\end{lemma}

\begin{remark}
\label{rem:nonmon}
If  $\delta=0$ and the controlled solution remains in $[K_1, K_3]$, it monotonically converges to either $K_1$ or $K_3$.
Also, in Assumption \ref{as:4eq1}  the function $g$ is allowed to have an infinite left-side local Lipschitz constant (derivative) at $K_1$ and right-side at $K_3$, while the case $\delta>0$
includes the situation when $g$ is differentiable at $K_1$ and $K_3$.

In the case $\delta>0$ and $\alpha<L/(L+1)$, there is, generally, no monotonicity of a solution, and to get desired stability it is not enough for the solution to remain in $[K_1-\delta, K_3+\delta]$, there could be a cycle inside of $(K_1-\delta, K_3+\delta)$. As expected, in this case the control  can be more flexible, and the parameter $\underline \alpha_0$ can be chosen smaller.
\end{remark}

The next lemma shows that sometimes the invariant set $(K_1-\delta, K_2)\cup (K_2, K_3+\delta)$ can actually be extended to $(K_0, K_2)\cup (K_2, \infty)$ without increasing the low bound $\underline \alpha_0$.
\begin{lemma}
\label{lem:alpha01}
Let Assumption \ref{as:4eq1} hold, $\underline \alpha_0$ be defined as in \eqref{def:alpha0}, $\alpha_*\in (\underline \alpha_0, 1)$, $\alpha^*\in (\alpha_*, 1)$ and $\alpha_n\in [\alpha_*, \alpha^*]$, for each $n\in \mathbb N$.
 Assume also that for all $\alpha>\underline \alpha_0$ we have
 \begin{equation}
\label{cond:or1}
\mbox{$\sup_{x\in (K_{0}, K_1-\delta)}G(\alpha, x)<K_3+\delta$ or $\inf_{x\in (K_{3}+\delta, \infty)}G(\alpha, x)>K_1-\delta$.}
 \end{equation}
 Then, for each $x_0>0$, $x_0\neq K_2$,
a solution to equation \eqref{eq:PBCvar} converges to an equilibrium. 
\end{lemma}

Lemma \ref{lem:alpha01} demonstrates that if we want to keep stability of $K_1$ and $K_3$ for all $x_0>0$, $x_0\neq K_2$, we need to increase the low bound of control only if \eqref{cond:or1} fails. So without loss of generality, we assume that
\begin{equation}
\label{cond:1}
g_m:=\max_{x\in [K_0,  K_1-\delta]}g(x)>K_3+\delta \,\, \mbox{and} \, \, g_m^2:=\min_{x\in [K_3+\delta,  g_m]}g(x)<K_1-\delta.
\end{equation}
Condition \eqref{cond:1} and Lemma \ref{lem:Gv} imply that, for each $\alpha\in (0,1)$,
\begin{equation*}
\label{cond:2}
\max_{x\in [K_0,  K_1-\delta]}G(\alpha, x)\le g_m \quad \mbox{and} \, \, \min_{x\in [K_3+\delta,  g_m]} G(\alpha,x)\ge g_m^2.
\end{equation*}
In the next section,  for each $\alpha>\underline \alpha_0$,  we extend the interval $(K_1-\delta, K_3+\delta)$ keeping the property of stability of $K_1$ and $K_3$, and then introduce the smallest $\alpha$, for which infinite circulation of solution between $(K_0, K_1-\delta)$ and $(K_3+\delta, g_m)$ becomes impossible.

\subsection{Sequences of attracting sets}
\label{subsec:trap2}
For $\delta$ as in Assumption \ref{as:4eq1} and $\beta\in [0, 1)$, define now  $\underline\kappa(\beta)$, the largest point of maximum  of $G(\beta, \cdot)$ on $(K_{0}, \, K_1-\delta)$, and $\bar \kappa(\beta)$,  the smallest point of minimum of $G(\beta, \cdot)$ on $(K_3+\delta, \, g_m)$:
\begin{equation}
\label{def:kappa12mr}
\begin{split}
&\underline\kappa(\beta)=\sup\left\{ y\in  \left(K_{0}, \, K_1-\delta\right): G(\beta, y)=\sup_{x\in (K_{0}, K_1-\delta)}G(\beta, x) \right\},\\
&\bar \kappa(\beta)= \inf \left\{y\in  \left(K_3+\delta, \, g_{m}\right):   G(\beta, y)=\inf_{x\in (K_3+\delta, g_m)}G(\beta, x) \right\}.
\end{split}
\end{equation}
Now we introduce two convergent sequences of points, $(d_n(\beta))_{n\in \mathbb N}$ and $(c_n(\beta))_{n\in \mathbb N}$, located in $(K_0, K_1-\delta)$ and in $(K_3+\delta, \, g_m)$, respectively,
\begin{equation}
\label{def:dc4eq}
\begin{split}
&d_0:=K_{1}-\delta, \, c_0:=K_{3}+\delta,\,\,  \mbox{and, for} \,\,  k\in \mathbb N,\\
&d_k(\beta):=\inf\left\{x\in \left(K_{0}, \, d_{k-1}\right): \max_{y\in [x, d_{k-1}] }G(\beta, y)\le  c_{k-1}  \right\},  \\
&d_k=d_{k-1}, \quad \mbox{if} \,\max_{y\in [K_0, d_{k-1}] }G(\beta, y)\le  c_{k-1},\\
&c_k(\beta):=\sup\left\{x\in \left(c_{k-1}, \, g_m\right):  \min_{y\in [c_{k-1}, x] }G(\beta, y) \ge  d_k(\beta) \right\}, \\ &c_k=c_{k-1}, \quad \mbox{if} \quad \inf_{y\in [c_{k-1}, g_m) }G(\beta, y)\ge  d_{k}.
\end{split}
\end{equation}
Assume that  $k$ is the first moment for which $d_k=d_{k-1}$. Since $c_k \ge c_{k-1}$, $
  \max_{y\in [K_{0}, d_{k}] }G(\beta, y)= \max_{y\in [K_{0}, d_{k-1}] }G(\beta, y)\le  c_{k-1}\le c_k$, we have
  $d_{k+1}=d_k$, therefore the sequence $(d_n(\beta))_{n\in\mathbb N}$ stops after $k-1$ steps. As $ \min_{y\in [c_{k}, x] }G(\beta, y)\ge  \min_{y\in [c_{k-1}, x] }G(\beta, y)$ for $x\in \left(c_{k}, \, g_m\right)$, we get
  \[
c_{k+1}(\beta):=\sup\left\{x\in \left(c_k, \, g_m\right): \!\!\! \min_{y\in [c_k, x] }G(\beta, y)\ge
\!\!\! \min_{y\in [c_{k-1}, x] }G(\beta, y) \ge d_{k}=d_{k+1}\right\}=c_k,
\]
therefore the sequence $(c_n(\beta))_{n\in \mathbb N}$ stops after $k$ steps.

Define the set of positive integers for which the iterative procedure stops
\begin{equation}
\label{def:0k0}
\begin{split}
&\mathcal S(\beta):=\{k: d_k(\beta)=d_{k+1}(\beta) \,\, \mbox{or} \,\, c_k(\beta)=c_{k+1}(\beta)\},\\
&k_0(\beta):=\inf\{\mathcal S(\beta)\}, \quad \mbox{if} \quad \mathcal S(\beta) \neq\emptyset, \quad  k_0(\beta):=\infty, \,\, \mbox{if} \,\, \mathcal S(\beta)=\emptyset.
\end{split}
\end{equation}
The condition
\begin{equation}
\label{def:k0}
\begin{split}
k_0(\beta):=\inf\{k: d_k(\beta)=d_{k+1}(\beta) \,\, \mbox{or} \,\, c_k(\beta)=c_{k+1}(\beta)\}<\infty
\end{split}
 \end{equation}
is essential in constructing a wider attractive set, dependent on the control $\beta$.

The sequences $(c_n(\beta))_{n\in \mathbb N}$ and $(d_n(\beta))_{n\in \mathbb N}$ are strictly monotone until a moment $k_0(\beta)\le \infty$, therefore the following limits exist
\begin{equation}
\label{def:limdc}
\hat c(\beta)=\lim_{n\to \infty} c_n(\beta), \quad \hat d(\beta)=\lim_{n\to \infty} d_n(\beta).
\end{equation}
 We can  define  the interval $(\hat d(\beta), \hat c(\beta))$, which includes $(K_{1}-\delta, K_{3}+\delta)$ and  is  invariant under $G(\beta, \cdot)$. When  $k_0(\beta)=\infty$,  $\{ \hat d(\beta), \hat c(\beta) \}$ is a two-cycle. Indeed, when  \eqref{def:k0} does not hold, we proceed to the limit, as $k\to \infty$, in the equalities $G(\beta, d_k(\beta))=c_{k-1}(\beta)$ and $G(\beta, c_k(\beta))=d_{k}(\beta)$ and get $G(\beta, \hat d(\beta))=\hat c(\beta)$, $G(\beta, \hat c(\beta))=\hat d(\beta).$
 So, for this particular $\beta$ the interval of the initial values  with the desired convergence cannot be increased.
Moreover, the bound for control $\alpha$ is sharp: if it is smaller, a cycle rather than an equilibrium can be an attractor.

All the above is summarized in the following lemma.
\begin{lemma}
\label{lem:propseqcd}
Let Assumption~\ref{as:4eq1} hold,
$\underline \alpha_0$ be defined as in \eqref{def:alpha0}, sequences $(d_n)_{n\in \mathbb N}$, $(c_n)_{n\in \mathbb N}$  and numbers $\underline\kappa$, $\bar \kappa$  by \eqref{def:dc4eq}, and \eqref{def:kappa12mr}, $k_0$ by \eqref{def:0k0}, and $\beta\in \left(\underline \alpha_0, \, 1\right)$.
 \begin{enumerate}
  \item [{\rm (i)}] For each $1\le k=k(\beta)< k_0(\beta)\le \infty$ we have: $d_k(\beta)<x<K_{1}-\delta \implies G(\beta, x)<c_{k-1}(\beta)$ and  $c_k(\beta)>x>K_{3}+\delta \implies G(\beta, x)>d_{k}(\beta)$.

 \item  [{\rm (ii)}] For each $1\le k(\beta)< k_0(\beta)\le \infty$ we have: $G(\beta, \cdot): [d_k(\beta), c_{k-1}(\beta)]\to  [d_k(\beta), c_{k-1}(\beta)]$.

\item  [{\rm (iii)}] The sequences $(c_n(\beta))_{n\in \mathbb N}$ and $(d_n(\beta))_{n\in \mathbb N}$ are strictly monotone until the moment $k_0(\beta)\le \infty$,
 $c_n(\beta)\uparrow \hat c(\beta)=\lim_{k\to \infty} c_k(\beta)$ and $d_n(\beta)\downarrow \hat d(\beta)=\lim_{k\to \infty} d_k(\beta)$.

 \item  [{\rm (iv)}] $\hat c(\beta)\in (K_{3}+\delta,  \bar \kappa], \,\, \hat d(\beta)\in [\underline \kappa, \, K_{1}-\delta)$ and $G(\beta, \cdot): [\hat d(\beta), \hat c(\beta)]\to  [\hat d(\beta), \hat c(\beta)]$.

  \item  [{\rm (v)}] If  \eqref{def:k0} holds then sequences $(d_n(\beta))_{n\in \mathbb N}$ and $(c_n(\beta))_{n\in \mathbb N}$ stop after either $k_0$ or $k_0+1$ steps, while if  \eqref{def:k0} fails then   $\{ \hat c(\beta), \hat d(\beta) \}$, as defined in \eqref{def:limdc},  is a two-cycle for the function $y=G(\beta, x)$.

 \item [{\rm (vi)}]  For any $\tilde \beta_1>\tilde \beta_2>\underline \alpha_0$  and each $k\in \mathbb N$, we have $ d_k(\tilde \beta_2)>d_k(\tilde \beta_1)$ and $ c_k(\tilde \beta_1)>c_k(\tilde \beta_2)$.
   \end{enumerate}
      \end{lemma}
      The  next lemma is an extension of  Lemma \ref{lem:delta01}, it proves convergence of a solution
to an equilibrium when the initial value is in the interval dependent on $\underline \alpha_0$, which is an extension of   $(K_{1}-\delta, K_{3}+\delta)$. Note that a solution to \eqref{eq:PBCvar} converges to $K_0$ or $K_2$ only in the cases of constant solutions with $x_0=K_0$ or $x_0=K_2$, respectively.

\begin{lemma}
\label{lem:x0beta0}
Let Assumption \ref{as:4eq1} hold, $\underline \alpha_0$ be defined as in \eqref{def:alpha0}, $ \alpha_*\in (\underline \alpha_0, 1)$,
$\alpha^*\in (\alpha_*, 1)$, $\alpha_n\in [\alpha_*, \alpha^*]$ for all $n\in \mathbb N$.
Let $\hat d(\alpha_*)$, $\hat c(\alpha_*)$  be denoted as in \eqref{def:limdc}.
Then  any solution  to \eqref{eq:PBCvar} with   $x_0\in (\hat d(\alpha_*), \hat c(\alpha_*))$ converges  to an equilibrium. 
\end{lemma}

Part (v) of  Lemma~\ref{lem:propseqcd} states that when $k_0(\alpha)=\infty$, the function $G(\alpha, \cdot)$ has a
two-cycle, so the condition  $ \alpha\in (\underline \alpha_0, 1)$  might not be sufficient for convergence of all solutions to an equilibrium.   Example~\ref{ex:infder} in Section~\ref{sec:ex} illustrates this.  In order to fix this problem, we introduce a lower bound $\underline \alpha$ for $\alpha$ which can be
larger than $\underline \alpha_0$
\begin{equation}
\label{def:underalphacd}
\underline \alpha:=\inf \left\{ \beta \in (\underline \alpha_{0},1): \max_{x\in [K_{0}, K_{1}]}G(\beta, x) <  \hat c(\beta) ~ \mbox{or} \!\!\!\!\inf_{y\in (K_{3}, g_m]}G(\beta, y) > \hat d(\beta) \right\},
\end{equation}
where $\hat c(\beta)$ and $\hat d(\beta)$ are defined as in \eqref{def:limdc}.
To see that  the set in \eqref{def:underalphacd} is non-empty, we note first that   $\hat c(\beta)\in (K_{3},  \bar \kappa], \,\, \hat d(\beta)\in [\underline \kappa, \, K_{1})$, see Lemma \ref{lem:propseqcd}~(iv). When $\beta$ is close to 1, $ \max_{x\in [K_{0}, K_{1}]}G(\beta, x) \approx K_{1}<K_{3}\le \hat c(\beta)$ and $\inf_{y\in (K_{3}, g_m) }G(\beta, y) \approx K_{3} >K_{1}\ge \hat d(\beta)$, so such $\beta$ belongs to the set in the right-hand side of \eqref{def:underalphacd}.  Applying Lemma~\ref{lem:propseqcd}~(vi), we conclude that each  $\alpha\in (\underline \alpha, 1)$ belongs to the set inside of the braces in \eqref{def:underalphacd}.

The following lemma highlights relations between $\underline \alpha_0$, $k_0(\alpha)$ and $\underline \alpha$.
 \begin{lemma}
\label{lem:k0}
Let $\underline \alpha_0$, $k_0(\beta)$ and $\underline \alpha$ be defined by \eqref{def:alpha0},
\eqref{def:0k0} and \eqref{def:underalphacd}.
 \begin{enumerate}
 \item  [(i)] $\underline \alpha_0=\underline \alpha$ if and only if   for  each $\alpha>\underline \alpha_{0}$, condition \eqref{def:k0} holds.
    \item  [(ii)] $\underline \alpha_0<\underline \alpha$  if and only if  there exists $\alpha>\underline \alpha_0$ such that  condition \eqref{def:k0} does not hold, i.e. $k_0(\alpha)=\infty$.
      \end{enumerate}
      \end{lemma}
			
We proceed now to the main result of this section.

\begin{theorem}
\label{thm:detcd}
Let Assumption~\ref{as:4eq1} hold, $\underline \alpha$ be defined as in \eqref{def:underalphacd}, $ \alpha_*\in (\underline \alpha, 1)$, $\alpha^*\in (\alpha_*, 1)$, $\alpha_n\in [\alpha_*, \alpha^*]$ for all $n\in \mathbb N$. Let $x$ be a
solution  to \eqref{eq:PBCvar} with $x_0 \in (K_0,\infty)$. Then $x$ converges  to an equilibrium. 
\end{theorem}

\begin{remark}
\label{rem:exG}
Based on results of Section \ref{subsec:trap2}, we conclude that $\underline \alpha$ defined as in \eqref{def:underalphacd} is the best deterministic lower bound for the control, which provides attraction of the solution  to either $K_1$ or $K_3$. Even though in general it is not so easy to find it, in some situations  it is possible.  In Section~\ref{subsec:exG} of the Appendix we demonstrate that when  function $g$ is differentiable outside of $[d_1(0), c_1(0)]$ we can find the lower threshold $\underline \alpha_1$ which is calculated based on $\underline \alpha_0$ and derivatives of $g$. Here $d_1(0)$ and $c_1(0)$ are computed by \eqref{def:dc4eq} for $g(x)$ instead of $G(\beta, x)$ and sometimes can be found easily.
\end{remark}

\subsection{Stochastically perturbed control}
\label{subsec:stoch}

We start by introducing a complete filtered probability space $(\Omega, {\mathcal{F}}$, $\{{\mathcal{F}}_n\}_{n \in
\mathbb N}, {\mathbb P})$, where the filtration $(\mathcal{F}_n)_{n \in \mathbb{N}}$ is naturally generated by
the sequence of independent identically distributed random variables $(\xi_n)_{n\in\mathbb{N}}$, i.e.
$\mathcal{F}_{n} = \sigma \left\{\xi_{1},  \dots, \xi_{n}\right\}$.
The standard abbreviation ``a.s.'' is further used for either ``almost sure" or ``almost surely"
with respect to a fixed probability measure $\mathbb P$, and
``i.i.d.'' for  ``independent identically distributed'', to describe random variables.
For details of stochastic concepts and notations we refer the reader to \cite{Shiryaev96}.

In many real-world models, in particular, in population dynamics, it is natural to assume that noises are bounded, which we describe in the following assumption, later a noise amplitude will be introduced.

\begin{assumption}
\label{as:noise}
$(\xi_n)_{n\in \mathbb N}$ is a sequence of independent identically distributed  random variables satisfying $|\xi_n|\le 1$, $\forall n\in \mathbb N$. Moreover, for each $\varepsilon>0$, $\mathbb P \{\xi\in (1-\varepsilon, \, 1] \}>0$.
\end{assumption}

We consider a control perturbed by  an additive noise,  $\alpha_n=\alpha+\ell \xi_{n+1}$,
\begin{equation}
\label{eq:PBCstoch}
x_{n+1}=g(x_n)-(\alpha+\ell \xi_{n+1})(g(x_n)-x_n), \quad x_0\in \mathbb R, \quad n\in {\mathbb N}_0.
\end{equation}
Here $\alpha \in (0,1)$, random variables  $\xi_n$  satisfy Assumption \ref{as:noise},  $\ell>0$ is a noise amplitude.

Let $ \underline \alpha_0$ and $\underline \alpha$ be defined as in \eqref{def:alpha0} and \eqref{def:underalphacd}, respectively.
In this section  we decrease  the lower bound $\underline \alpha$ for a control parameter $\alpha$  proposed in the previous sections, applying stochastic perturbations. Set
\begin{equation}
\label{def:alphaell1}
\begin{split}
&\alpha_n=\alpha+\ell\xi_n, \quad 
\alpha\in \bigl(0.5(\underline \alpha_0+\underline \alpha),  \,\, \underline \alpha\bigr), \quad \ell\in \bigl(\underline \alpha-\alpha, \,\, \min \{ 1-\alpha,\alpha-\underline \alpha_0\}\bigr).
\end{split}
\end{equation}
Since $\underline \alpha\ge \underline \alpha_0$, see Lemma \ref{lem:k0},  any $\alpha$ satisfying \eqref{def:alphaell1} will be smaller than $\underline \alpha$ which is the best lower estimate for the deterministic case.

The following lemma was proved in \cite{BKR,BRAllee} and is a corollary of the Borel-Cantelli Lemma.
\begin{lemma}
\label{lem:topor}
Let sequence $(\xi_n)_{n\in \mathbb N}$ satisfy Assumption \ref{as:noise}. Then, for each nonrandom $\mathbf J\in \mathbb N$, $\varepsilon>0$ and a random moment $\mathcal M$, there is a random moment $\mathcal N=\mathcal N(\mathbf J, \varepsilon, \mathcal M) \ge {\mathcal M}$, $\mathcal N \in {\mathbb N}$, such  that 
$
\mathbb P\{\xi_{\mathcal N+i} \ge 1-\varepsilon, \,\,  i=0, 1, \dots, \mathbf J  \}=1.
$
\end{lemma}

\begin{theorem}
\label{thm:stoch2sides}
Let Assumptions \ref{as:4eq1} and \ref{as:noise}  hold,
$\underline \alpha_0$  and $\underline \alpha$ be defined as in \eqref{def:alpha0} and  \eqref{def:underalphacd}, $\alpha$ and $\ell$ satisfy \eqref {def:alphaell1}. Then a solution  $x$ to \eqref{eq:PBCstoch} with $x_0\in (K_0, \infty)$ converges to one of the equilibrium points 
with the  total probability one.
\end{theorem}

\section{An arbitrary number of equilibrium points}
\label{sec:arbeq}
In this section we consider a continuous function $g$ which might have more than four equilibrium points, see the fourth iterate of Ricker's map in Fig.~\ref{fig:Rick4} as an example. An approach similar to previous sections is applicable, however it becomes quite technical, therefore we reduce considerations to the analogue  of Assumption~\ref{as:4eq1} when  $\delta=0$. The general case we leave for the future research.

\begin{figure}[ht]
\centering
\includegraphics[height=.24\textheight]{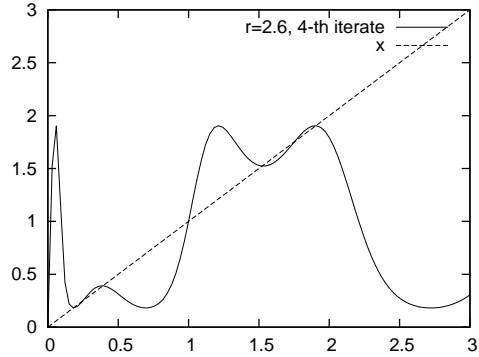}
\caption{The fourth iteration of the Ricker map for $r=2.6$.}
\label{fig:Rick4}
\end{figure}

We aim to stabilize every second point only, so we define the number of these points as
\begin{equation}
\label{def:i0}
i_0:=\left[\frac{ j_0}2-1\right], \quad I:=\left\{0, 1, \dots, i_0  \right\},
\end{equation}
where $\left[t \right] $ is the largest integer not exceeding $t$.
\begin{assumption}
\label{as:LL1}
Let $g:\mathbb R\to \mathbb R$ be a continuous function such that
\begin{equation}
\label{eq:equil1}
\begin{split}
&g(K_j)=K_j\in \mathbb K:=\{K_j, \,  j=0, 1,2, \dots, j_0\}, \,\, K_{j_0}\le \infty,\\
& g(x)-x\ge 0, \, \mbox{for} \,\,  x\in (K_0,K_1),
\\ & g(x)-x \mbox{ alternates its sign on adjacent } \,  (K_j,K_{j+1}), \\
&\, j=0, 1,2, \dots, j_0-1, \,\, \mbox{$g(x)-x$ can vanish for some $x\in (K_j,K_{j+1})$},\\
&x\ge g(x)\ge b, \,\, \mbox{for some} \,\, b\in \mathbb R, \,\, \mbox{when}\,\, x\in (K_{j_0-1}, \infty), \\ & \mbox{if} \,\, K_{j_0}=\infty, \, j_0 \, \mbox{is even}.
\end{split}
\end{equation}
\end{assumption}
The last two lines in \eqref{eq:equil1} correspond to higher death rates for overpopulation, which is satisfied for
any iterates of the Ricker, logistic and other population dynamics maps. In the chaotic case, these maps are characterized
by much higher global Lipschitz constants on one of two adjacent segments with a common potentially stabilizable point,
which corresponds to a one-side Lipschitz condition 
at each such point in Assumption~\ref{as:LL1},
\begin{equation}
\label{cond:Lrl}
\begin{split}
&K_{2i+1}-g(x)\leq  L^+_{2i+1} (x-K_{2i+1}), ~ x\in (K_{2i+1}, K_{2i+2}),~ i\in I,\\
&g(x)-K_{2i+1} \leq  L^-_{2i+1} (K_{2i+1}-x), ~  x\in (K_{2i}, K_{2i+1}),~ i=0, 1, \dots, \left[\frac{ j_0-1}2\right].
\end{split}
\end{equation}
An efficient control is possible, only if at least one of one-side Lipschitz constants at each point $K_{2i+1}$ is finite.

\begin{assumption}
\label{as:LLj0}
Let $g$ from Assumption~\ref{as:LL1} satisfy that either $j_0$ is even,
or $g$ is left locally Lipschitz at $K_{j_0}<\infty$, and
$\min \left\{L^-_{2i+1} , \,  L^+_{2i+1} \right\}<\infty$,  $i \in I$, where $L^+_{2i+1}, L^-_{2i+1} \in (0, \infty]$ are defined in \eqref{cond:Lrl}.
\end{assumption}

Under Assumptions~\ref{as:LL1} and \ref{as:LLj0}, for $K_{j_0} < \infty$,  either $j_0$ is even, or $g$ is left locally Lipschitz at $K_{j_0}$.
In the former case, $g(x)<x$ for $x \in (K_{j_0-1}, K_{j_0})$, and, due to continuity of $g$, the function
$(g(x)-K_{j_0})/(K_{j_0}-x)$ is continuous on $[K_0,K_{j_0-1}]$ and
$$
L^-_{j_0} := \max_{x \in [K_0,K_{j_0}]} \frac{g(x)-K_{j_0}}{K_{j_0}-x}=\max_{x \in [K_0,K_{j_0-1}]} \frac{g(x)-K_{j_0}}{K_{j_0}-x} < \infty.
$$
In the latter case, if $L^*$ is a local Lipschitz left constant, $g(x)-K_{j_0} \leq L^* (K_{j_0}-x)$, $x\in (K_{j_0}-\varepsilon,
K_{j_0})$, we again use continuity of $(g(x)-K_{j_0})/(K_{j_0}-x)$ on $[K_0,K_{j_0}-\varepsilon]$ to get
$$
L^-_{j_0} := \max _{x \in [K_0,K_{j_0}]} \frac{g(x)-K_{j_0}}{K_{j_0}-x}=\max \left\{ L^*,
\max_{x \in [K_0,K_{j_0}-\varepsilon]} \frac{g(x)-K_{j_0}}{K_{j_0}-x} \right\}< \infty.
$$
For $K_0$, $g(x)>x\geq K_0$, $x\in [K_0,K_1]$, the function $(g(x)-K_0)/(x-K_0)$ is continuous on $[K_1,K_{j_0}]$, therefore
$$ L^+_{0} := -\min_{x \in [K_0,K_{j_0}]} \frac{g(x)-K_0}{x-K_0}=- \min_{x \in [K_1,K_{j_0}]} \frac{g(x)-K_0}{x-K_0} < \infty.$$

Summarizing, under Assumptions~\ref{as:LL1} and \ref{as:LLj0}, when $K_{j_0}<\infty$,  there exist finite positive constants $L^-_{j_0}$ and $L^+_{0}$ such that
\begin{equation}
\label{cond:-Kj0}
\begin{split}
&g(x)-K_{j_0} \leq  L^-_{j_0} (K_{j_0}-x), \quad  x\in (K_0, K_{j_0}), \quad  L^-_{j_0}<\infty,\\
&g(x)-K_{0} \geq - L^+_{0} (x-K_0), \quad  x\in (K_0, K_{j_0}), \quad  L^+_{0}<\infty.
\end{split}
\end{equation}

\begin{remark}
Once we are free to choose $j_0$ in the control scheme, the condition concerning the value of $j_0$ in Assumption~\ref{as:LLj0} is
not really a limitation, assuming local one-side Lipschitz condition at each $K_{2i+1}$ is sufficient.
Whenever $g:(K_0, \, K_{j_0})\to (K_0, \, K_{j_0})$, both inequalities in \eqref{cond:-Kj0} hold
with arbitrarily small constants $L^-_{j_0}$ and $L^-_{0}$.
\end{remark}

Denote
\begin{equation}
\label{def:L2i+1}
L_{2i+1}:=\min\left\{  L^-_{2i+1}, \,  L^+_{2i+1} \right\},    ~i\in I, \quad \bar L:=\max\left\{ L^+_0, \, L^-_{j_0}, \, L_{2i+1}, \,i\in I\right\},
\end{equation}
where the notations in \eqref{cond:Lrl} and \eqref{cond:-Kj0} are used. Without loss of generality, we assume that $L>1$.
Now, introduce the sets
\begin{equation}
\label{def:I12}
\begin{split}
I^+:=\left\{i\in I: L_{2i+1}=L^+_{2i+1}\right\},\quad I^-:=\left\{i\in I: L_{2i+1}=L^-_{2i+1} < L^+_{2i+1} \right\},
\end{split}
\end{equation}
i.e. for $i\in I^+$ we have $L^+_{2i+1}\le L^-_{2i+1}$, and $L^+_{2i+1} > L^-_{2i+1}$ 
for  $I^-$. Also, $I^+\cup I^-=I$ 
and $I^+\cap I^-=\emptyset$.

Similarly to Lemma~\ref{lem:LL+2},  we can justify the following result.

\begin{lemma}
\label{lem:comb}
Let  $g$ from Assumption \ref{as:LL1} satisfy conditions \eqref{cond:Lrl} and \eqref{cond:-Kj0}, $G$ be defined as in \eqref{def:G}.
\begin{enumerate}
\item  [(i)] For $a:=\max\left\{L^-_{j_0}/(1+L^-_{j_0}), \, \,  L^+_{0}/(1+L^+_{0})\right\}$, $\alpha \in (a,1]$,  we have $G(\alpha, \cdot): (K_0, K_{j_0})\to (K_0, K_{j_0}).$

\item [(ii)] Let $\alpha\in\bigl(\bar L/(\bar L+1), 1\bigr)$, and $x_n$ be a solution to \eqref{eq:PBCvar} with $\alpha_n\in  (\bar L/(\bar L+1), 1)$.
\begin{enumerate}
\item   If $x_0\in (K_{2i}, K_{2i+1})$, $i\in I^-$ then
$G(\alpha, x_0)\in  (K_{2i}, K_{2i+1})$, and $x_n$ is non-decreasing and converges  
either to $K_{2i+1}$ or to some equilibrium in  $(K_{2i}, K_{2i+1})$.

\item   If $x_0\in (K_{2i+1}, K_{2i+2})$, $i\in I^+$ then
$G(\alpha, x_0)\in  (K_{2i+1}, K_{2i+2})$, and $x_n$ is non-increasing and converges either to $K_{2i+1}$ or to some equilibrium in  $(K_{2i}, K_{2i+1})$.

\item  If $i\in I^+$ and $i+1\in I^-$, then
$G(\alpha, \cdot): [K_{2i+1}, K_{2i+3}]\to [K_{2i+1}, K_{2i+3}]$.
Moreover, if  $x_0\in (K_{2i}, K_{2i+2})$ then $x_n$ converges to either $K_{2i+1}$, or $K_{2i+3}$, or an equilibrium in $(K_{2i+1}, K_{2i+2})\cup (K_{2i+2}, K_{2i+3})$ and is monotone.
 \end{enumerate}
\end{enumerate}
\end{lemma}

Everywhere in this paper we assume  $\alpha_n>\bar L/(\bar L+1)$, which makes it possible to apply Lemma~\ref{lem:comb}.  Definition of sets in~\eqref{def:I12} allows us to distinguish between $K_{2j+1}$ depending on the side where the Lipschitz constant  is smaller, i.e. whether $L_{2i+1}= L^+_{2i+1}$ or $L_{2i+1}= L^-_{2i+1}$. Similarly to the four-equilibrium case, we can establish stability of equilibria  $K_{2j+1}\in \mathbb K$ with odd indexes only. Lemma~\ref {lem:comb} describes the behavior of a solution when $x_0\in  (K_{2i}, K_{2i+1})$, $i\in I^-$,  and $x_0\in  (K_{2i+1}, K_{2i+2})$, $i\in I^+$. However, when $x_0\in  (K_{2i}, K_{2i+1})$, $i\in I^+$,  and $x_0\in  (K_{2i+1}, K_{2i+2})$, $i\in I^-$, the situation is different: the solution can get out of the interval, start traveling from one interval to another including infinite circulation between some  intervals.

The general idea contains packing the intervals surrounding a potentially stabilizable equilibrium
into blocks, based on whether the left or the right Lipschitz constant is bigger. At the very beginning, we apply a control which makes open segments with a smaller constant monotone attractors.
Thus, if we have a group $\mathcal V_{0}$ of consecutive segments with prevailing right constants, a solution with an initial point in this group can only be switching to the segments to the left, and therefore converges to one of the equilibrium points in $\mathcal V_{0}$. Similarly,
for a group $\tilde {\mathcal V}$ of consecutive segments with prevailing left constants, any solution with an initial point in this group can only be switching to the segments to the right, and therefore converges to one of the equilibrium points in $\tilde{\mathcal V}$.
Our main analysis is devoted to the block $ \mathcal V_{2s+1}$, $s=0, \dots, s_0-1$,
consisting of two back-to-back groups of both types,  between which  circulation is possible.  We introduce two-step control process: first keeping all solutions inside a designated block, and then eliminating circulation between two groups inside each block. For brief  illustration of intervals we refer to Fig.~\ref{fig:Rick4}, where we have $L_1^->L_1^+$, $L_3^-<L_3^+$, $L_5^->L_5^+$, $L_7^-\!<\!L_7^+$, $\tilde{\mathcal V}=\mathcal V_0=\emptyset$, $\mathcal V_1=\{(K_0, K_1), \, (K_1, K_2),\, (K_2, K_3),\,(K_3, K_4)\}$, $\mathcal V_3=\{(K_4, K_5), \, (K_5, K_6),\, (K_6, K_7),\,(K_7, \infty)\}$.

The next theorem is  the main result for the deterministic control on all $(K_0, K_{j_0})$,
when only one-side Lipschitz constants at odd-numbered equilibrium points are supposed to be finite. Most details on classification of the intervals along with the proof of Theorem~\ref{thm:det} are deferred to the Appendix.
Note that in Theorem~\ref{thm:det} constant solutions with $x_0=K_{2j}$ 
are excluded from consideration, see the remark before Lemma~\ref{lem:x0beta0}.

\begin{theorem}
\label{thm:det}
Let Assumptions \ref{as:LL1} and \ref{as:LLj0}   
hold, $I$  be defined as in \eqref{def:i0}. Then there exists a control  $ \alpha_*\in \bigl(\bar L/(\bar L+1), 1\bigr)$ such that for each  $\alpha^*\in (\alpha_*, 1)$, if $\alpha_n\in [\alpha_*, \alpha^*]$ for all $n\in \mathbb N$, a solution  $x$ to \eqref{eq:PBCvar} with $x_0\in (K_0, K_{j_0})$  converges  either to some $K_{2i+1}$, $i\in I$,  or to any equilibrium inside  $(K_p, K_{p+1})$, $0\le p\le j_0-1$.
\end{theorem}

\begin{remark}
\label{rem:inf}
The results of Theorem~\ref{thm:det} can be extended to an infinite number of equilibrium points, $\mathbb K=\{K_i, \, i\in \mathbb N \}$ with $\lim_{i\to \infty} K_i=\infty$, if we consider a finite subinterval.
For each $p\in \mathbb N$ and $x_0\in (K_0, K_p)$, Theorem~\ref{thm:det} holds with some $\alpha=\alpha(p)$.
\end{remark}
In some situations, when a control is perturbed by  an additive noise, $\alpha+\ell \xi_{n}$, Theorem~\ref{thm:det} can be improved by decreasing the lower bound on $\alpha$. However, in the case of an arbitrary number of equilibrium points, even the statement  of all necessary conditions is quite involved, so we defer most of this part to the Appendix.


\section{Examples and simulations}
\label{sec:ex}

In this section we consider three examples illustrating our results. In particular, simulations demonstrate that
introduction of noise into a deterministic control sometimes extends the range of $\alpha$ which
guarantees stability.

\begin{example}
\label{ex:Ricker2}
Consider the second iteration of the Ricker map $g=f^2$, $r=2.7$, as illustrated by Fig.~\ref{fig:Ricker2}.
The function $g=f^2$  is infinitely differentiable and has 3 positive fixed points $K_1\approx 0.214$, $K_2=1$,  $K_3 \approx 1.786$. The map $g=f^2$ has a minimum at the point $\frac{1}{r} \approx 0.37\in (K_1, K_2)$   
 with the value $g_{\min}\approx 0.142$, and monotone derivatives on each intervals, $(K_1, K_2)$ and  $(K_2, K_3)$, and it has two maximums with the same value $g_{\max}=\frac{e^{r-1}}{r}\approx 2.027$. We estimate Lipschitz constant $L$ on these intervals as $L\approx 1.65$. The left Lipschitz constant on $(K_0, K_1)=(0, 0.214) $ is quite large, it exceeds 9.8, leading to $(L-1)/(L+1)>0.8$. Thus application of results of \cite[Theorem 2.5]{LizPBC2012}  gives us a lower estimate for $\alpha_n$ exceeding 0.8, taking into account that the right Lipschitz constant only gives $(L-1)/(L+1) \approx 0.2452$.

We show that  $\beta\approx 0.249$ belongs to the right-hand side of \eqref{def:alpha0} for some $\delta>0$.  The corresponding function $G(\beta, \cdot)$ takes its minimum on $(K_1, K_2)$ at $x_{\min~\beta}$,
which can be found from  the equation $g'(x_{\min\beta})=-\beta/(1-\beta)$.   Approximating $g'(x)$ numerically and taking into account
$g'(x_{\min\beta}) \approx -0.331$,  we get $x_{\min~0.249}\approx 0.3387$.

For $x\in (K_1, K_2)$ and $\beta=0.249$, we estimate
\begin{align*}
&G(\beta, x)\ge  (1-\beta)g_{\min}+\beta x_{\min \, \beta}>-\delta+K_1, \\
& \delta> K_1-(1-\beta)g_{\min}-\beta x_{\min \, \beta}
 \approx 0.022,
\end{align*}
which is consistent with the graph on Fig.~\ref{fig:Ricker2}.
So the inequality on the second line of \eqref{def:alpha0} holds. Since $|G'(\beta, x)|$ is
decreasing  for $x>K_3$, we do not need to take care about the inequality on the third line of \eqref{def:alpha0}.

For $\beta=0.249$ we have $\max_{x\in (0, K_1)}G(\beta, x)\le (1-\beta)\max_{x\in (0, K_1)} g(x)+\beta K_1 \le 1.786=K_3<K_3+\delta$,
so, by \eqref{def:dc4eq}, \eqref{def:0k0},   we have $d_1(\beta)=d_0(\beta)=\hat d(\beta)=K_1-\delta$, $k_0(\beta)=1$, and, by \eqref{def:underalphacd}  we have $\beta>\underline \alpha$.
Thus, Theorem~\ref{thm:detcd} implies attractivity of $K_1$ and $K_3$.

A bifurcation diagram on Fig.~\ref{bif_Ric_squared} (left) confirms the result.
\begin{figure}[ht]
\centering
\includegraphics[height=.18\textheight]{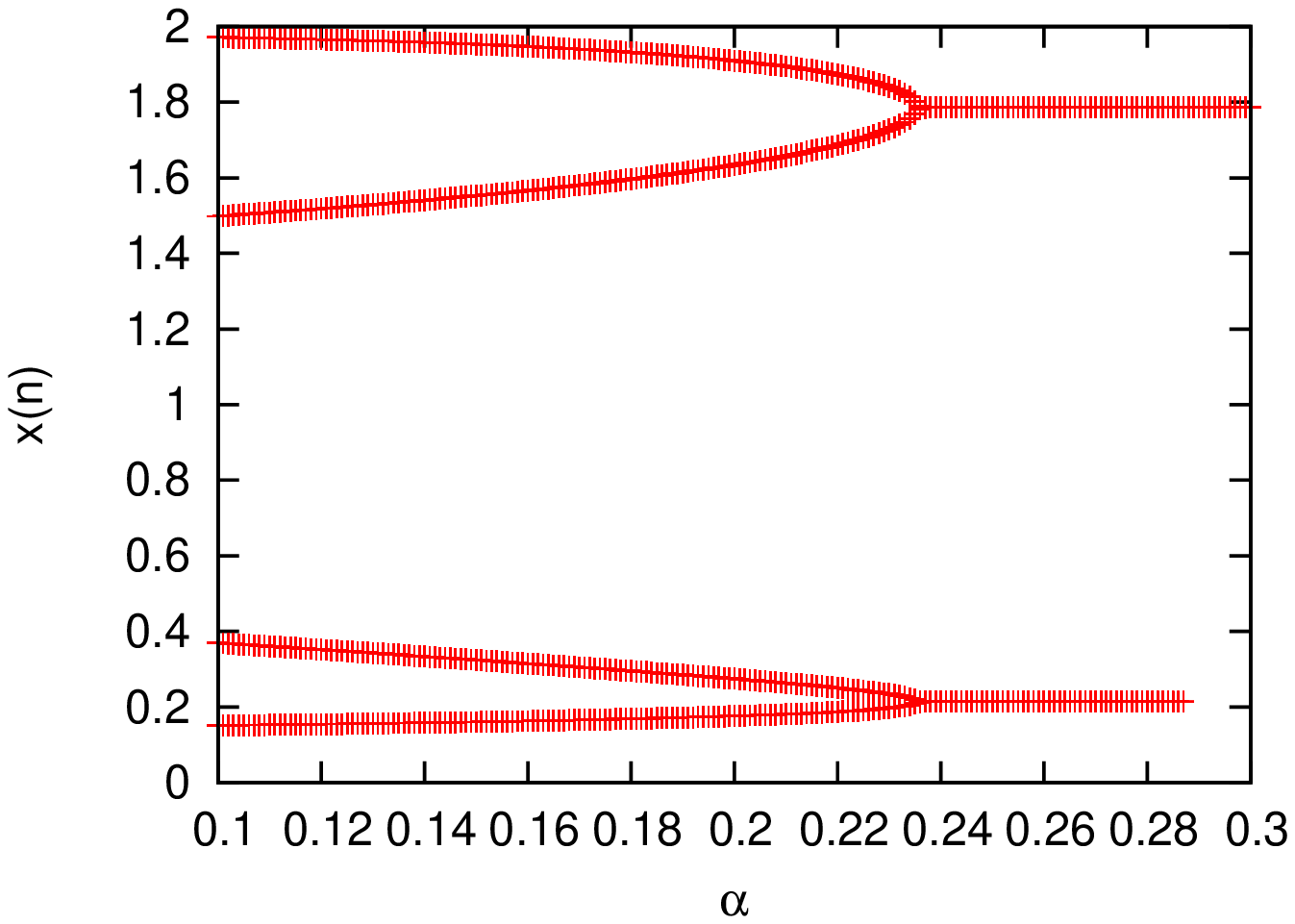}
~~
\includegraphics[height=.18\textheight]{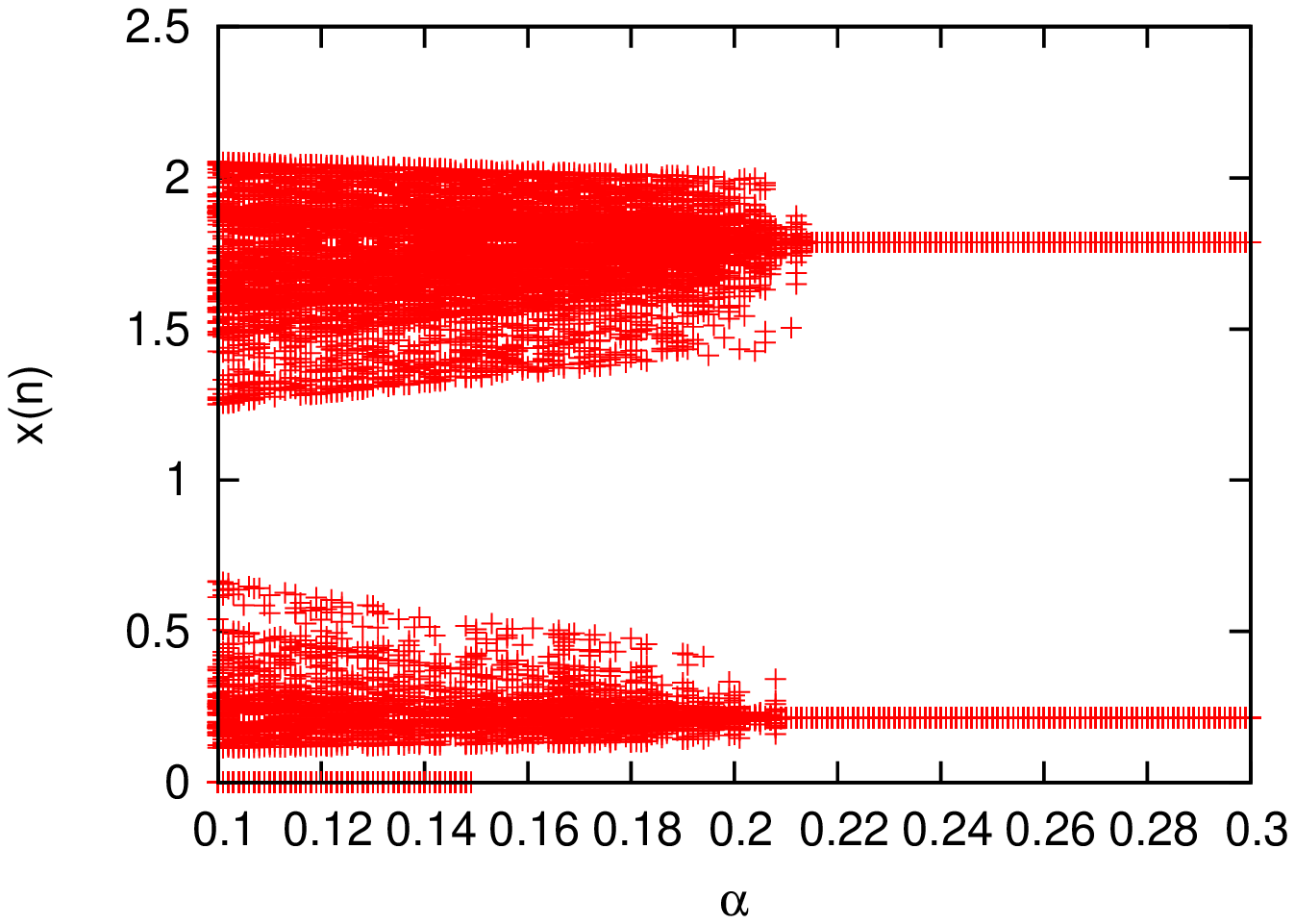}
\caption{A bifurcation diagram for the second iterate of the Ricker  map  with $r=2.7$, $\alpha \in
(0.1,0.3)$ and (left) no noise, (right) $\ell=0.15$.
}
\label{bif_Ric_squared}
\end{figure}

When a control is stochastically perturbed  by the noise with the amplitude $\ell=0.15$,
the bifurcation diagram in Fig.~\ref{bif_Ric_squared} (right) demonstrates that stabilization is achieved at
$\alpha\approx 0.215$,
which is smaller than for the deterministic case $\alpha\approx 0.235$, see Fig.~\ref{bif_Ric_squared} (left), and is aligned with the result of Theorem~\ref{thm:stoch2sides}.
\end{example}

\begin{example}
\label{ex:Ricker3}
Consider the third iteration of the Ricker map $g=f^3$, $r=3.5$,
which has 7 equilibrium points, $K_i$,  $K_8:=\infty$,
$|g'(K_{2i+1})|\le 15.62=L$, $i=0, 1, \dots, 7$, and the right Lipschitz constants $L^+_{2i+1}$ on  the intervals $(K_{2i+1}, K_{2i+2})$, $i=0, \dots 3$
do not exceed $L \approx 15.62$. However, some of left Lipschitz constants $L^-_{2i+1}$ on  the intervals $(K_{2i}, K_{2i+1})$ are greater than 300.

Applying  the comments from Section~\ref{sec:arbeq} and classification from Section~\ref{subsubsec:classif}, we conclude that all the
intervals
$(K_p, K_{p+1})$, $p=1,2, \dots,7$ belong to the block  \,$\tilde {\mathcal V}$.
So for $\alpha_n\ge 0.94$, the circulation of solution between intervals is impossible, and all the equilibria $K_{2i+1}$, $i=0, \dots, 3$
are stable.
The bifurcation diagram in Fig.~\ref{Ricker3bifur} (left) demonstrates a slightly smaller value of $\approx 0.88$. And again, this number is reduced
to
$\approx 0.86$ if a noise with $\ell=0.06$ is introduced, see  Fig.~\ref{Ricker3bifur} (right).
\begin{figure}[ht]
\centering
\includegraphics[height=.16\textheight]{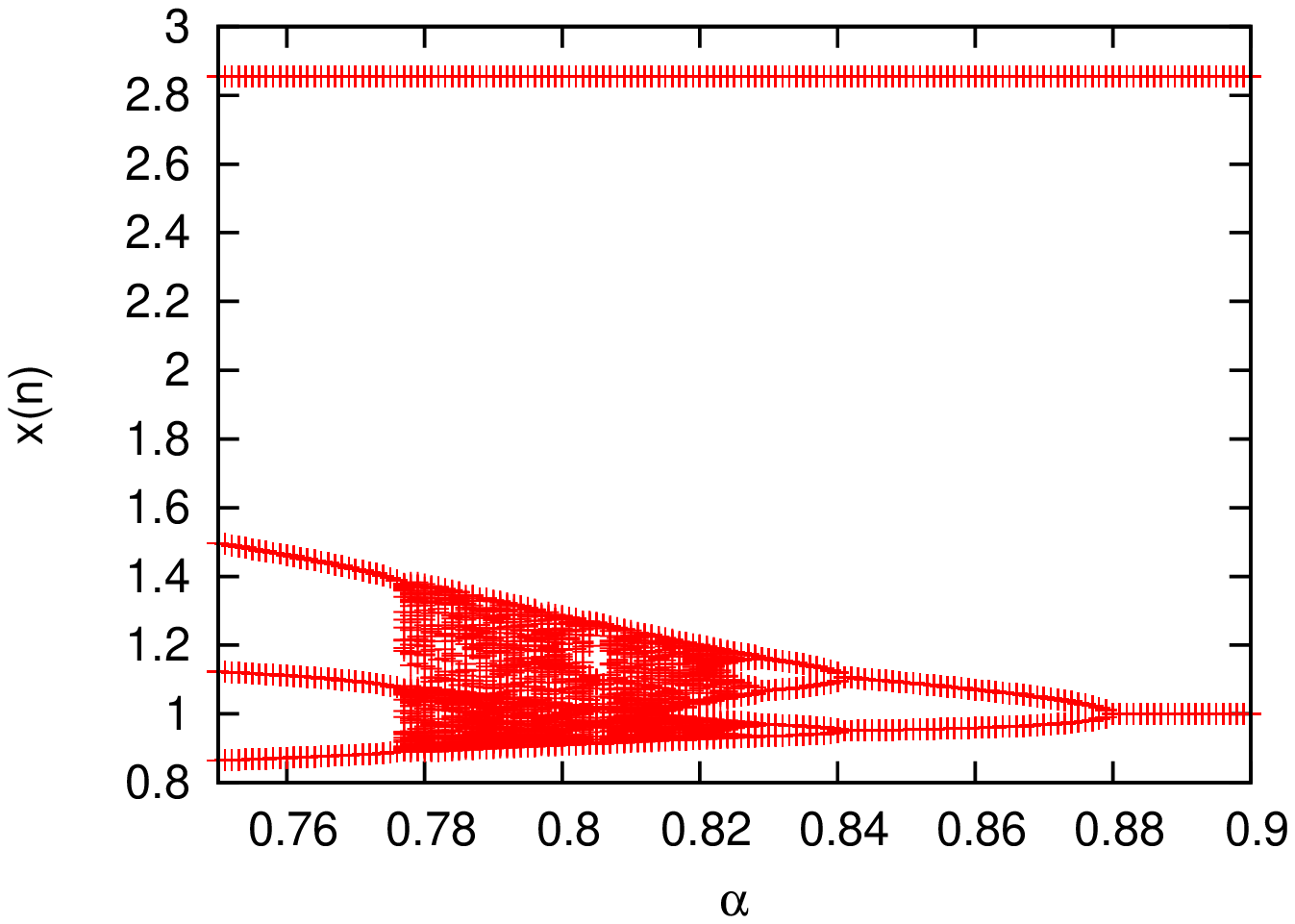}
~~
\includegraphics[height=.16\textheight]{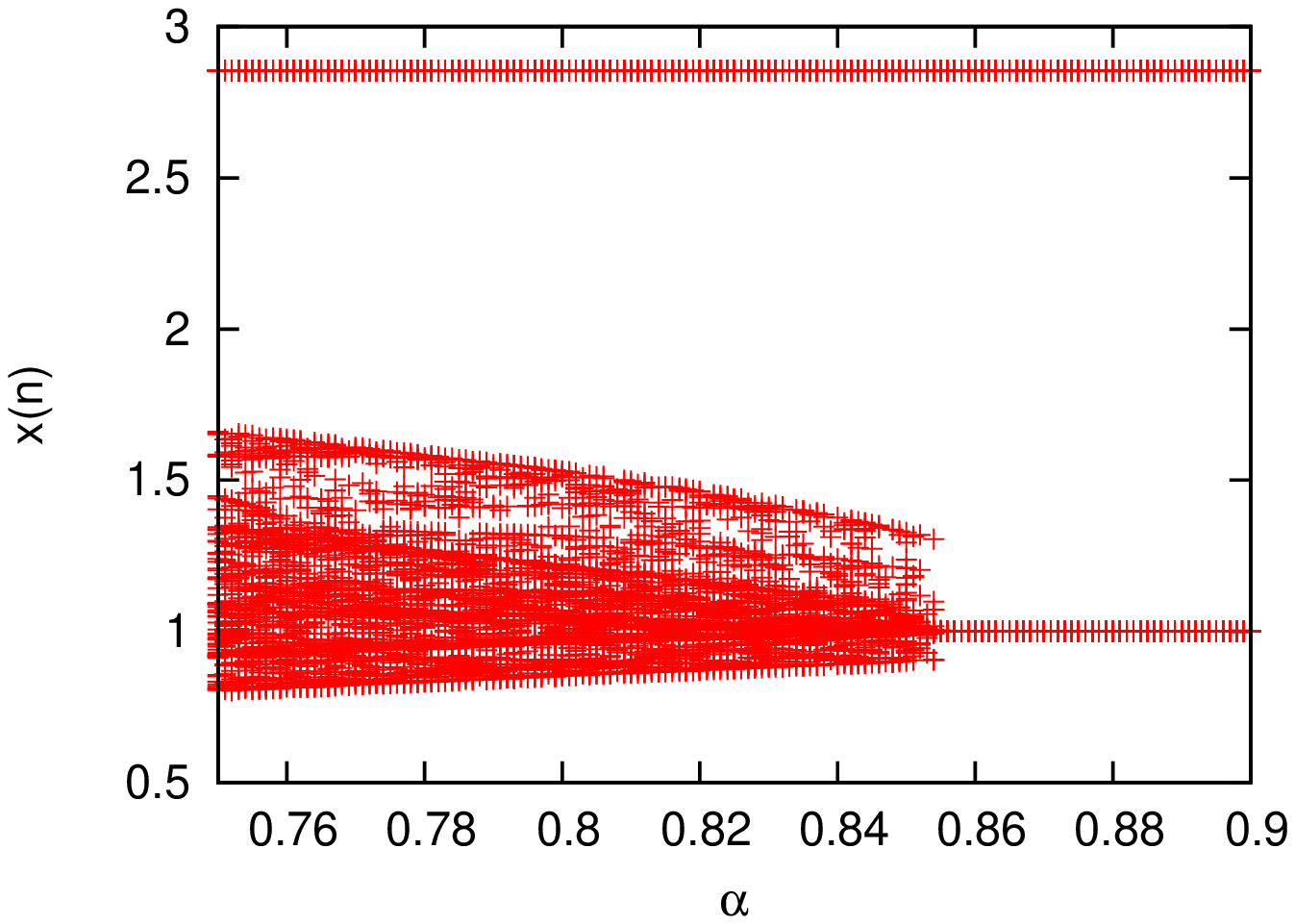}
\caption{A bifurcation diagram for the third iterate of the Ricker map
with $r=3.5$ for $\alpha\in (0.75,0.9)$ and (left) without
noise, (right) $\ell=0.06$. The last bifurcation
leading to two stable equilibrium points
occurs for $\alpha \approx 0.88$ in the deterministic case
and $\alpha<0.86$ in the stochastic case.}
\label{Ricker3bifur}
\end{figure}
\end{example}

\begin{example}
\label{ex:infder}
This example illustrates the results of Section~\ref{sec:arbeq}.
Define
\begin{equation}
\label{ex:1}
g(x)=\left \{
\begin{array}{cccc}
  0,& x\in (-\infty, 0],
\vspace{2mm}
\\
\frac{163}{63}x,& x\in (0,\, 0.9],
\vspace{2mm}
\\
x+\frac{0.6}{0.42\sqrt{0.1}}\sqrt{1-x},& \quad  x\in (0.9, \, 1],
\vspace{2mm}\\
x-\frac{5}{23\pi}\sin[10\pi(x-1)], & x\in (1, 1.2],
\vspace{2mm} \\
x-\frac{0.6}{0.42\sqrt{0.3}}\sqrt{x-1.2},& \quad  x\in \left(1.2,\, 1.5\right],
\vspace{2mm}
\\
 2x-\frac{41}{14},& \quad  x\in \left(1.5, \frac{41}{14}\right],  \vspace{2mm} \\
 x+10\sqrt{(x-\frac{41}{14})(3-x)},& \quad  x\in \left(\frac{41}{14}, 3\right],   \vspace{2mm} \\
x-\sqrt{x-3} ,& \quad  x\in (3,\infty),
  \end{array}
\right.
\end{equation}
which has 6 equilibrium points, see Fig.~\ref{Fig_graph}.
\begin{figure}[ht]
\centering
\includegraphics[height=.24\textheight]{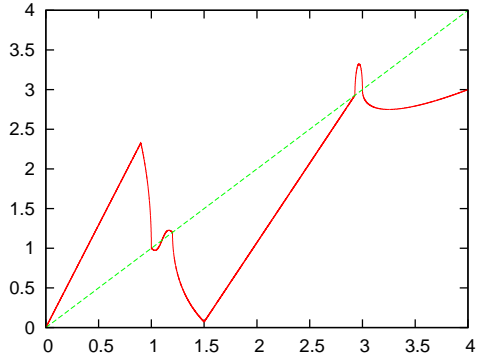}
\caption{The graph of the map defined in \protect{\eqref{ex:1}}, together with $y=x$.
}
\label{Fig_graph}
\end{figure}
Here both left- and right-side derivatives at $K=3$ are infinite, so for efficient control, we have to reduce ourselves to the segment $[0, 41/14]$ getting the set of equilibrium points
$\mathbb
K=\left\{0, 1, \, 1.1, \, 1.2,\, 41/14\right\}$, 
$K_4=41/14<\infty$.
The equilibrium points $K_1=1$ and $K_3=1.2$ are unstable, since $g'(1^-)=-\infty=g'(1.2^+)$ and $g'(1^+)=-27/23=L=g'(1.2^-)$, so $\beta_0= \underline
\alpha=L/(1+L)=0.54$. By \eqref{ex:1}, for   each $\beta$, the point of minimum of $G(\beta, x)$ is $\underline \kappa=0.9$, and the point of maximum is $\bar\kappa=1.5$ and $g:[0, \, 41/14]\to [0, \, 41/14]$.

By straightforward calculation we show that for each $\beta\in (0.54, \, 0.6040)$, the function $G(\beta, x)$
has two unstable 2-cycles, located in $[0.9; \,1)$ and $[1.2, \, 1.5)$, respectively.
For example, for $G(0.58, x)$, there are two two-cycles at approximately $\{0.9, \, 1.5 \}$ and
$\{0.9749, \, 1.2755 \}$.
For $d_k$, $\hat d$, $c_k$, $\hat c$ defined by  \eqref{def:dc4eq}
and \eqref{def:limdc} (when $\beta=0.58$) we have
$0.9<0.9749\le\hat d < d_k<K_1=1$, $1.2=K_3<c_k <\hat  c\le 1.2755<1.5$, $k\in\mathbb N$,
so the equilibrium points $1$ and $1.2$ are unstable.
For $\beta=0.6040$ there is only one 2-cycle, $\approx \{0.95, \, 1.3497 \}$. For  $\beta>0.6040$ there is no
two-cycle.

Fig.~\ref{Fig_inf_nn} (left)  presents a bifurcation diagram for $G(\alpha, \cdot)$
with $g$ defined in \eqref{ex:1}
and demonstrates efficient stabilization in the chosen segment $[0, \, 41/14]$ and attractors in addition to the two equilibrium points for $\alpha < 0.605$, as the theory predicts.
Fig.~\ref{Fig_inf_nn} (right)  contains a bifurcation diagram for  $G(\alpha, \cdot)$,
when the control $\alpha$ is perturbed by the Bernoulli noise with $\ell=0.04$, and
two stable equilibrium points starting with a smaller $\alpha\approx 0.535$.

\begin{figure}[ht]
\centering
\includegraphics[height=.18\textheight]{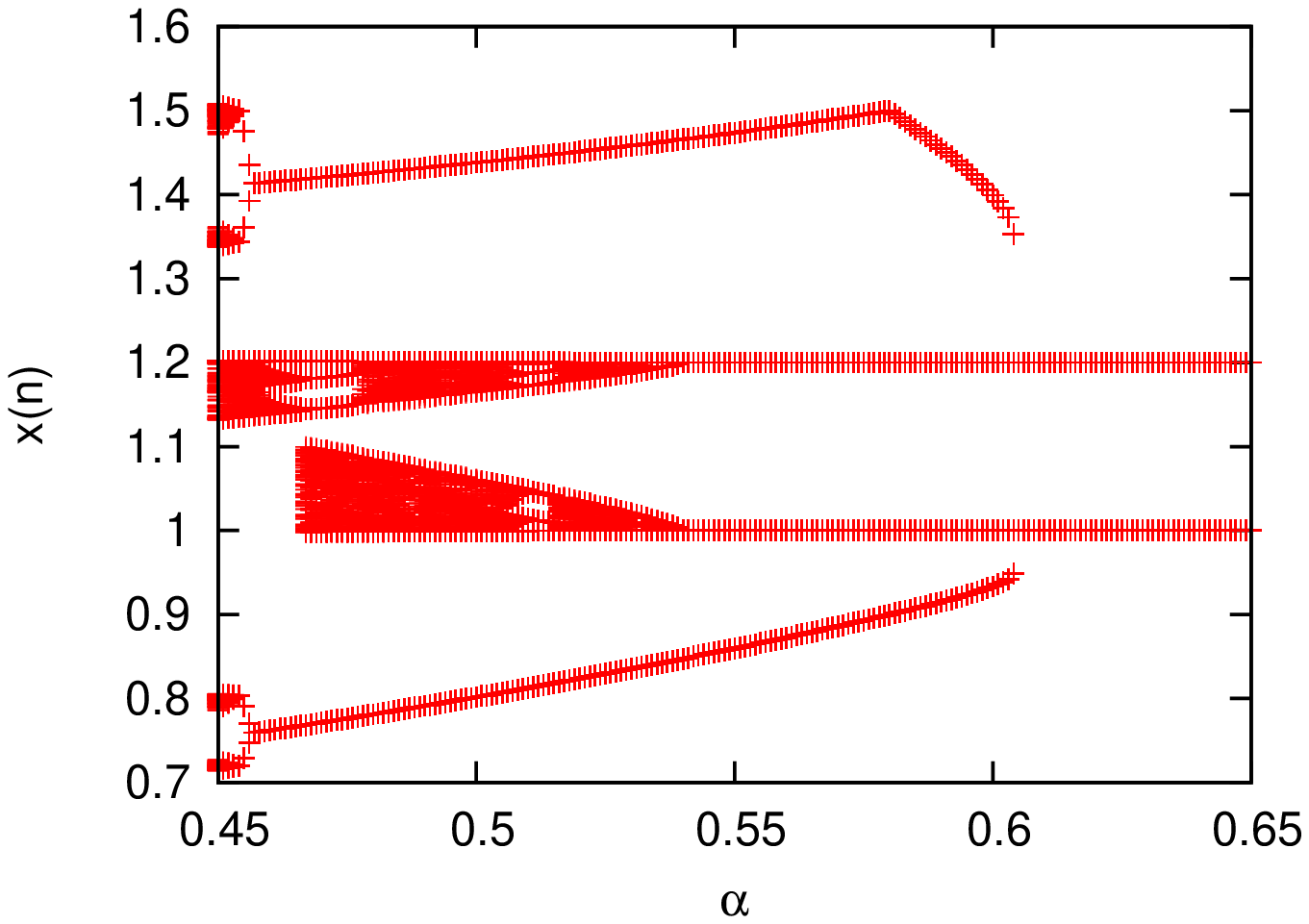}
\includegraphics[height=.18\textheight]{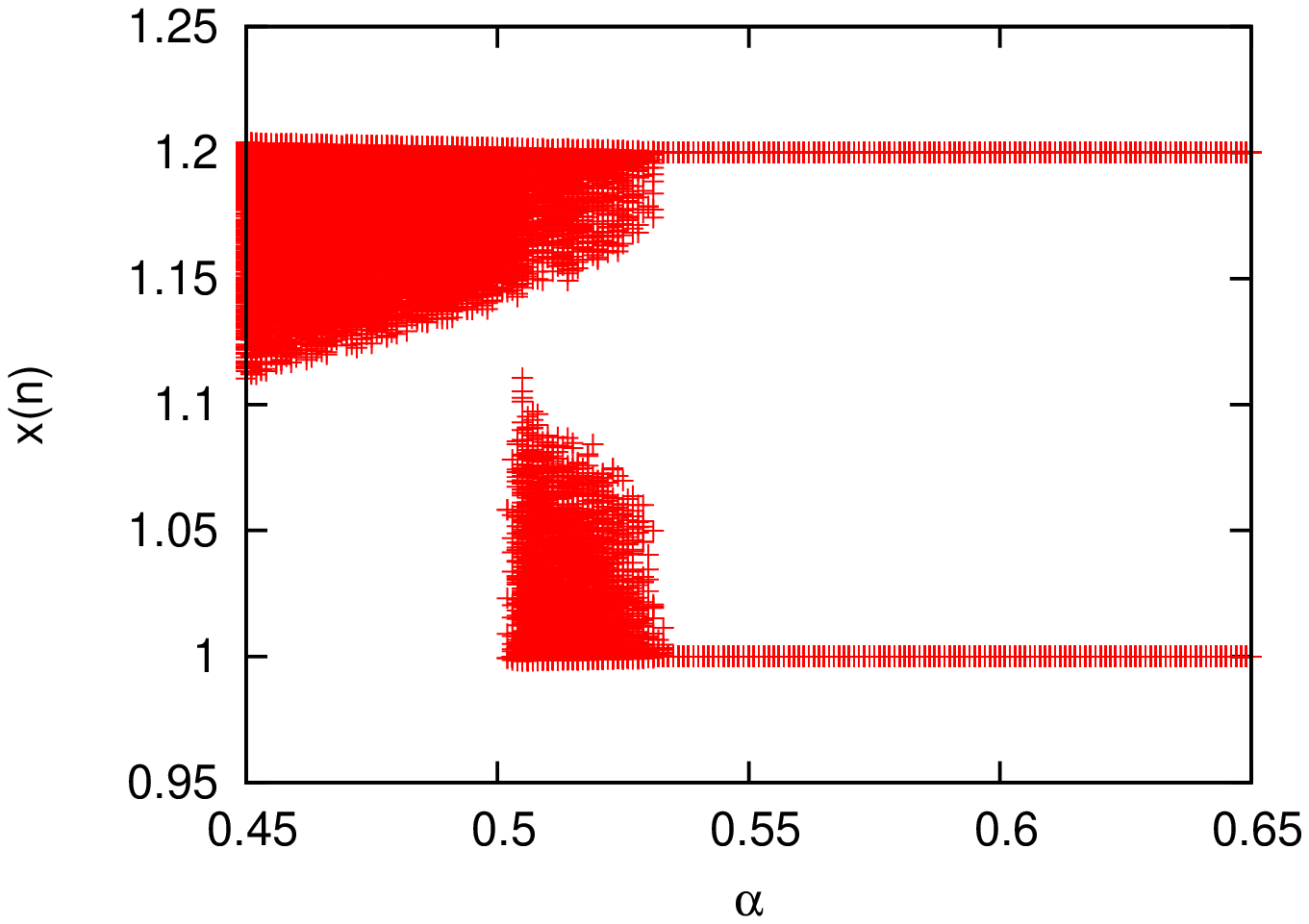}
\caption{A bifurcation diagram for the map defined in \protect{\eqref{ex:1}} with $\alpha \in
(0.45,0.65)$ and (left) no noise, we have two stable equilibrium points starting from $\alpha\approx 0.605$,
(right) for Bernoulli noise with $\ell=0.04$, the last bifurcation happens for smaller $\alpha\approx 0.535$.
Here the two attractors correspond to two stable equilibrium points with separate basins of attraction, not to a cycle.
}
\label{Fig_inf_nn}
\end{figure}
\end{example}


\section{Discussion and future research}
\label{sec:disc}

For stabilization of maps with PBC, if a map is unimodal with a negative Schwarzian derivative, the best control bound is easily computed using the derivative at an equilibrium \cite{FL2010}. If a map is not unimodal and has several fixed points, global Lipschitz two-side constants can be used to compute a required control bound in a similar way \cite{LizPBC2012}. However, for very large one-side Lipschitz constants, this bound is quite close to one and is far from being optimal. When the difference between two one-side constants is significant, it may be preferable to choose a control in such a way that, once a point gets into a smoother area, it stays there forever.
Unfortunately, for models with many fixed points, in particular, for iterates of common maps, this approach does not guarantee stabilization, as a solution can wander between ``steep'' areas. The algorithm to choose an optimal control bound, together with making it lower in the stochastic case, is the main accomplishment of the present paper. Specifically, we managed
\begin{itemize}
\item
to overcome the problem of large global Lipschitz constants compared to a derivative at an equilibrium by considering one-side constants only, which allows to relax control conditions;
\item
to include the case when a one-side derivative is infinite (in this case the constant $L/(L+1)$ is sharp);
\item
introducing a noise in the control, to relax restrictions on the average control intensity even more.
\end{itemize}

PBC was proved to be an efficient tool in stabilizing multiple equilibrium points, even in the case when the continuous map is not smooth.
Introduction of noise allowed to lower the level of average control, and this result complements \cite{BKR}. Compared to \cite{BKR2020},
the results of the present paper are global, not local, which gives an advantage in practical implementations.

Some possible extensions of the present work are listed below.

\begin{itemize}
\item
Section~\ref{sec:4eq} can be extended to the case of $\delta>0$ and an arbitrary number of equilibrium points.
Also, the results for iterates of maps can be stated in the terms of the original function and a pulse PBC control.

\item
It would be interesting to generalize our results to systems of difference equations subject to PBC control \cite{LP2014}, for example, relating to \cite{LP2014} in a stochastic case and combining \cite{Nag2016,Porfiri_SIAD} with PBC-type control.

\item
Following the previous item, results for systems can be readily applied to
models of population dynamics, controlling competition outcomes.
The essential role of stochasticity was recently analyzed in  \cite{Schreiber1,Schreiber}, see also references therein.
\end{itemize}


\section*{Acknowledgments}
The authors are very grateful to anonymous reviewers whose thoughtful comments significantly contributed to the quality of presentation of our results.



\section{Appendix}

\subsection{Proofs for sections \ref{sec:4eq}}
\label{sec:6_1}

\subsubsection{Proof of Lemma \ref{lem:conv}}

We have $x_{n+1}-x_n=(1-\alpha_{n+1})g(x_n)-(1-\alpha_{n+1})x_n=(1-\alpha_{n+1})(g(x_n)-x_n).$
So
$
|g(x_n)-x_n| = |x_{n+1}-x_n|/(1-\alpha_{n+1})\le |x_{n+1}-x_n|/(1-a)\to 0, \quad \text{as} \quad m\to \infty,
$
and, by continuity of $g$, we have
$
g(x^*)-x^*=\lim_{n\to \infty}[g(x_n)-x_n]=0.
$

\subsubsection{Proof of Lemma \ref{lem:Gv}}
Parts (i)-(iii) are immediate results of definition~\eqref{def:G} and continuity of $G$.
To prove (iv),  we note that since $1-\hat \mu=1-\frac{\mu-\mu_0}{1-\mu_0}=\frac{1-\mu}{1-\mu_0}$, we have $(1-\hat \mu)[(1-\mu_0)=1-\mu$ and  $\hat \mu=\mu-(1-\hat \mu)\mu_0$. Therefore,
\begin{align*}
G(\mu, x) & =(1-\hat \mu)(1-\mu_0)g(x) +\mu x
=(1-\hat \mu) G(\mu_0, x)-(1-\hat \mu)\mu_0 x+\mu x\\&=(1-\hat \mu) G(\mu_0, x) +[\mu-(1-\hat \mu)\mu_0] x=(1-\hat \mu) G(\mu_0, x) +\hat \mu x.
\end{align*}

\subsubsection{Proof of Lemma \ref{lem:LL+2}}
Note that $\alpha>L/(L+1) \implies L(1-\alpha)-\alpha<0$.

(i) For $x\in (K_1, K_2)$, by  Assumption~\ref{as:4eq1}, we get
\begin{align*}
K_2 & >x >G(\alpha, x)=(1-\alpha) (g(x)-K_1)+\alpha(x-K_1)+K_1
\\&\ge -L(1-\alpha) (x-K_1)+\alpha(x-K_1)+K_1
\\ & =-[L(1-\alpha)-\alpha](x-K_1)+K_1>K_1.
\end{align*}
The case $x\in (K_2, K_3)$ is similar.

(ii)  For $x\in (K_1-\delta, K_1)$, we have $K_1-\delta<x<G(\alpha, x)$ and
\[
G(\alpha, x)=(1-\alpha) (g(x)-K_1)+\alpha(x-K_1)+K_1\le [L(1-\alpha)-\alpha](K_1-x)+K_1<K_1.
\]
The case $x\in (K_3, K_3+\delta)$ is similar.

(iii) Parts (i) and (ii) imply that the solution $x_n$ remains in the same interval (either $(K_1-\delta, K_1)$ or $(K_1, K_3)$ or  $(K_3, K_3+\delta)$) as its initial value $x_0$.
In each of the three intervals the sequence $(x_n)_{n\in \mathbb N}$ is monotone and bounded, and therefore is convergent. The application of Lemma~\ref{lem:conv} completes the proof.

\subsubsection{Proof of Lemma \ref{lem:alpha0}}
 Take some $\alpha>L/(L+1)$. Applying Lemma~\ref{lem:LL+2}~(i)-(ii) we conclude that $
 \min_{x\in [K_1-\delta, K_1]}G(\beta, x)>K_1-\delta$ and $\min_{x\in [K_1, K_2]}G(\beta, x)>K_1.$
 Similar estimations can be done for $\max_{x\in [K_2, K_3+\delta]}G(\beta, x)$. So $\alpha$ belongs to the set on the second and third lines of \eqref{def:alpha0}, which, therefore, is non-empty. It also implies $\alpha\ge  \underline \alpha_0$, and proves Part~(ii).

\subsubsection{Proof of Lemma \ref{lem:delta01}}
 The case $\delta=0$  is covered by Lemma~\ref{lem:LL+2}~(i),(iii).  So we consider only $\delta>0$.

(i) Let $x\in (K_1-\delta, K_2)$, then $G(\alpha, x)>K_1-\delta$. If $x\in (K_1, K_2)$ then $G(\alpha, x)<x<K_2$ by \eqref{def:alpha0}. For $x\in (K_1-\delta, K_1)$,  since $\delta<K_2-K_1$,  we have
\begin{equation}
\label{est:delta1}
\begin{split}
G(\alpha, x)&=(1-\alpha) (g(x)-K_1)+\alpha(x-K_1)+K_1 \\ & \le L(1-\alpha) (K_1-x)+\alpha(x-K_1)+K_1\\&=[L(1-\alpha)-\alpha](K_1-x)+K_1< (x-K_1)+K_1<\delta+K_1<K_2.
\end{split}
\end{equation}
The case  $x\in (K_2, K_3+\delta)$ is similar.

(ii) Note that when $\underline \alpha_0=L/(L+1)$, we are under the assumptions of Lemma~\ref{lem:LL+2}. Assume that  $\underline \alpha_0<L/(L+1)$.
If $\alpha_*<L/(L+1)$, we set  $q:=L(1-\alpha_*)-\alpha_*$ and note that $q\in (0, 1)$ and $L(1-\alpha_{n})-\alpha_{n}<q$.
Let $x_n\in (K_1-\delta, K_1)$. Estimation~\eqref{est:delta1} gives us
$
G(\alpha_{n+1}, x_n)-K_1<[L(1-\alpha_{n+1})-\alpha_{n+1}](K_1-x_n)<q(K_1-x_n).
$
So either $G(\alpha_{n+1})$ remains in $(K_1-\delta, K_1)$ after some moment or it changes the position around $K_1$ infinitely many times. In both cases $\lim_{n\to \infty}x_n=K_1$.
If $\alpha_*=L/(L+1)$, we set $\underline \alpha_*:=0.5[\underline \alpha_0+L/(L+1)]\in \left(\underline \alpha_0,  L/(L+1)\right)$  and define  $q:=L(1-\underline\alpha_*)-\underline\alpha_*$. Since $\alpha_{n}\ge \alpha_*=L/(L+1)>\underline \alpha_*$, the rest of the proof is the same as above.


\subsubsection{Proof of Lemma \ref{lem:alpha01}}
We need to check only the case $x_0\in (K_0, K_1-\delta) \cup (K_3+\delta, \infty)$. Let $x_0\in (K_0, K_1-\delta)$, the other case is similar.
If the first relation in \eqref{cond:or1} holds, $x_n$ cannot remain in $(K_0, K_1-\delta)$ forever, since in that case it should converge to the equilibrium which is not in $(K_0, K_1-\delta)$. So it enters  $(K_1-\delta, K_3+\delta)$ in a finite number of steps, and we can apply Lemma \ref{lem:delta01}. If the second relation in \eqref{cond:or1} holds then $x_{\bar n}$  can get out to $(K_3+\delta, \infty)$ for some $\bar n\in \mathbb N$.
In this case  $x_n\in (K_1-\delta, x_{\bar n})$ for all $n\ge \bar n$, so $x_{n_0} \in (K_1-\delta,K_3+\delta)$
for some $n_0 \ge \bar n$, and we can apply Lemma~\ref{lem:delta01} again.

\subsubsection{Proof of Lemma \ref{lem:propseqcd}}
 Parts  (i)-(ii) follow from definitions \eqref{def:dc4eq} and the fact that $G(\beta, x)\ge x$ for $x\in (K_{0}, K_{1})$, and $G(\beta, x)\le x$ for $x\in (K_{3}, g_m)$, where $g_m$ is defined in \eqref{cond:1}.
Definitions \eqref{def:dc4eq}, Part (i) and Lemma \ref{lem:LL+2} imply Part~(iii).
Part   (iv) follows from \eqref{def:dc4eq}, definitions \eqref{def:kappa12mr} of $\underline \kappa(\beta)$, $\bar \kappa(\beta)$ and Part (ii).

The first part of (v) follows from definition \eqref{def:0k0}.
When  \eqref{def:k0} does not hold, we proceed to the limit as $k\to \infty$ in the equalities $G(\beta, d_k(\beta))=c_{k-1}(\beta)$ and $G(\beta, c_k(\beta))=d_{k}(\beta)$ and get $G(\beta, \hat d(\beta))=\hat c(\beta)$, $G(\beta, \hat c(\beta))=\hat d(\beta)$ by continuity.

To prove (vi), for $\tilde \beta_1>\tilde \beta_2$, we use that $G(\tilde \beta_2, x)>G(\tilde \beta_1, x)$ for $x\in  (K_{0}, K_{1})$,  and $G(\tilde \beta_2, x)<G(\tilde \beta_1, x)$ for $x\in  (K_{3}, g_m)$. Applying \eqref{def:dc4eq}  we show that $ d_1(\tilde \beta_2)>d_1(\tilde \beta_1)$. To prove $ c_1(\tilde \beta_1)>c_1(\tilde \beta_2)$, we use the inequalities
$
 \min_{y\in [c_0, x] }G(\tilde \beta_1, y) \ge  \min_{y\in [c_0, x] }G(\tilde \beta_2, y)\ge  d_1(\tilde \beta_2) >d_1(\tilde \beta_1).
 $
 The rest of the proof can be done inductively using the same approach.

\subsubsection{Proof of Lemma \ref{lem:x0beta0}}
By Lemma \ref{lem:delta01} we need to consider only $x\in (\hat d(\alpha_*), K_1-\delta)\cup  (K_3+\delta, \hat c(\alpha_*))$.  When $x_0\in (\hat d(\alpha_*), K_{1}-\delta)$, we have  $x_0\in (d_{k_1}(\alpha_*), K_{1}-\delta)$ for some $k_1\in \mathbb N$.
In this case $x_0<x_1=G(\alpha_1, x_0) <G(\alpha_*, x_0)<  c_{k_1-1}(\alpha_*)<  \hat c(\alpha_*)$,  and therefore, $x_1$ gets into either $(x_0, K_{1}-\delta)$ or $(K_{1}-\delta, K_{3}+\delta)$ or $(K_3+\delta, c_{k_1-1}(\alpha_*))$.  In the second case, a solution stays in  $(K_{1}-\delta, K_{3}+\delta)$ forever, see Lemma~\ref{lem:delta01}.  In the last case  $G(\alpha_*, x_1)> d_{k_1-1}(\alpha_*)$ by  Lemma \ref{lem:propseqcd}(i).

 So we are only interested in the situation when $x_1\in (K_{3}+\delta,  c_{k_1-1}(\alpha_*))$. We have  then $x_1>x_2=G(\alpha_2, x_1) >G(\alpha_*, x_1)> d_{k_1-1}(\alpha_*)$, and therefore, $x_2$ gets into either
$(K_{3}+\delta, x_1)$ or $(K_{1}-\delta, K_{3}+\delta)$ or $(d_{k_1-1}(\alpha_*), K_{1})$.  Circulation can happen only in the last case if  $x_3$ gets into
$(K_{3}+\delta, c_{k_1-2}(\alpha_*))$.  Applying  Lemma \ref{lem:propseqcd}~(i),
we conclude that a solution cannot make transition  between $(K_{3}+\delta, \, c_{k_1}(\alpha_*))$ and $(d_{k_1}(\alpha_*), K_{1}-\delta)$ more than $k_1+1$ times, and after each transition, a solution moves one level  closer  to $(K_{1}-\delta, K_{3}+\delta)$,
in the sense that now possible transitions are between  $(K_{3}+\delta, \, c_{k_1-1}(\alpha_*))$ and $(d_{k_1-1}(\alpha_*), K_{1}-\delta)$,
i.e. if before the transition a solution was in $(d_{i}(\alpha_*),\, c_{i}(\alpha_*))$, after the next round of transitions a solution is
in $(d_{i-1}(\alpha_*),\, c_{i-1}(\alpha_*))$. Therefore it either enters $(K_{1}-\delta, K_{3}+\delta)$ after a finite number of steps or remains in $(x_0, K_{1}-\delta)\cup (K_{3}+\delta, c_{k_1}(\alpha_*))$ and converges to either $K_1$ or $K_3$.

Similar reasoning is applied for $x_0\!\in \!(K_{3}+\delta, \hat c(\alpha_*))$.

\subsubsection{Proof of Lemma \ref{lem:k0}}

 (i) If $\underline \alpha=\underline \alpha_0$  we have
        $ \max\limits_{x\in [K_{0}, K_{1}]}G(\beta, x) <  \hat c(\beta)$ or
      $\min\limits_{y\in [K_{3}, g_m] }G(\beta, y) > \hat d(\beta)$ for any $\beta>\underline \alpha_0$ .
       Assume that $\max\limits_{x\in [K_{0}, K_{1}]}G(\beta, x) <  \hat c(\beta)$. Since $\hat c=\lim_{k\to \infty}c_k$, there exists a number $k_0$ s.t.
       $ \max\limits_{x\in [K_{0}, K_{1}]}G(\beta, x) <c_{k_0}<  \hat c(\beta)$.     By definition \eqref{def:dc4eq} this implies
       $d_{k+1}=d_{k}$ for $k\ge k_0+1$ and $\hat d(\beta)=d_{k_0}$. 		
			The case  $\min\limits_{y\in [K_{3}, g_m] }G(\beta, y) > \hat d(\beta)$ is similar.
       If condition \eqref{def:k0}  holds,  $\forall \beta\in (\beta_{0}, 1)$ we have $k_0(\beta)<\infty$. Suppose $d_{k}(\beta)=d_{k_0}(\beta)$, for all $k\ge k_0(\beta)$, then  $\max\limits_{x\in [K_{0}, K_{1}]}G(\beta, x) <  c_{k_0-1}(\beta)< \hat c(\beta)$, which implies $\beta_{1}=\beta_{0}$. The case  when $c_{k}(\beta)=c_{k_0}(\beta)$ for all $k\ge k_0$ is similar.

       (ii) If  $\underline \alpha>\underline \alpha_0$, there exists $\beta$, $\underline \alpha<\beta<\underline \alpha_0$ s.t. $\max_{x\in [K_{0}, K_{1}]}G(\beta, x)>\hat c(\beta)$ or\\ $\min_{x\in [K_{3}, g_m]}G(\beta, x)>\hat d(\beta)$. Assume the first inequality. Since $(c_n(\beta))_{n\in \mathbb N}$ is non-decreasing, we have $\max_{x\in [K_{0}, K_{1}]}G(\beta, x)> c_k(\beta)$  for each $k\in \mathbb N$, so \eqref{def:dc4eq}  implies that $(d_n(\beta))_{n\in \mathbb N}$ continues to decrease and does not stop. So $(c_n(\beta))_{n\in \mathbb N}$ also continues to increase and does not stop.   Then condition \eqref{def:k0}  does not  hold, which implies that
       $\bar\kappa(\beta)>\hat c(\beta)>c_k(\beta)$ and $\underline \kappa(\beta)<\hat d(\beta)<d_k(\beta)$ for all $k\in \mathbb N$.  So $\min_{y\in [K_{3}, g_m] }G(\beta, y) =G(\beta, \bar\kappa(\beta))<G(\beta, \hat c(\beta))=\lim_{k\to \infty}G(\beta, c_k)=\lim_{k\to \infty}d_k=\hat d(\beta)$. Analogously, $\max_{x\in [K_{0}, K_{1}]}G(\beta, x) >  \hat c(\beta)$.
Therefore $\beta<\underline \alpha$.


 \subsubsection{Proof of Theorem \ref{thm:detcd}}
Let  $x_0\in (K_0, K_1-\delta)$. Take some $\alpha_*>\underline \alpha$ and assume first that $\displaystyle \max_{x\in [K_{0}, K_{1}-\delta]}G(\alpha_*, x) ~< ~\hat c(\alpha_*)$. Since $\hat c(\alpha_*)$ is a limit of a non-decreasing sequence,  there exists $k_1$ s.t. $\max_{x\in [K_{0}, K_{1}-\delta]}G(\alpha_*, x) <  c_{k_1}(\alpha_*)<  \hat c(\alpha_*)$. From this place we just follow the scheme for the proof of Lemma \ref{lem:x0beta0}.
Similar approach applies for $x_0\!\in \!(K_{3}+\delta, g_m)$ and  $\inf_{x\in (K_{3}+\delta, g_m)}G(\alpha_*, x)>~\hat d(\alpha_*)$, as well as other cases.

\subsubsection{Proof of Theorem \ref{thm:stoch2sides}}
Fix some $\alpha$ and $\ell$ as in \eqref{def:alphaell1}. Note that it is enough  to consider only $x_0\in (g^2_m, g_m)$. By \eqref {def:alphaell1} we have $  \alpha-\ell>\underline \alpha_0, \,\alpha+\ell>\underline \alpha$, $\alpha+\ell>\alpha_n>\alpha-\ell$, $n\in \mathbb N$.
Fix some  $\displaystyle \varepsilon \in (0, 1-(\underline \alpha-\underline \alpha_0)/2)$  s.t. $\alpha+\ell>\underline \alpha+\varepsilon \ell$.
Set
\[
\alpha_*:=\alpha+(1-\varepsilon)\ell, \quad \alpha^*:=\alpha+\ell.
\]
Since $\alpha_*>\underline \alpha$ and by \eqref{def:dc4eq}, \eqref{def:underalphacd}, there exists $k_0=k_0(\alpha_*)$ s.t.
\begin{equation*}
\label{ineq:k1}
\sup_{x\in [K_{0}, K_{1}-\delta]}
\!\!\! G(\alpha_*, x) < c_{k_0}(\alpha_*) \le \hat c(\alpha_*) \mbox{ or } \inf_{y\in (K_{3}+\delta, g_m]}
\!\!\! G(\beta, y) > d_{k_0+1}(\alpha_*)>\hat d(\alpha_*).
\end{equation*}
Applying \eqref{def:dc4eq} and Lemma \ref{lem:propseqcd}~(i)-(ii), we conclude that any solution $x$ to \eqref{eq:PBCstoch} with $\alpha_n\equiv \alpha_*$ and $x_0\in  (K_{0}, K_{1}-\delta]\cup [K_{3}+\delta, g_m]$ either remains in $[g^2_m, K_{1}-\delta]$ or $ [K_{3}+\delta, g_m]$ or circulates between these two intervals. If it remains in  $[g^2_m, K_{1}-\delta]$, it will exceed $K_{1}-\delta$ in $N_1$ steps. Similarly, if the solution remains in $[K_{3}+\delta, g_m]$, it  will be less than $K_{3}+\delta$ in $N_2$ steps, where
$$
N_1:=\left[\frac{K_{1}-\delta-g^2_m}{\min_{x\in [g^2_m, K_{1}-\delta]}\{G(\alpha_*, x)-x \}}\right]+1, $$
$$ N_2:=\left[\frac{g_m-K_{3}-\delta}{\min_{x\in [K_{3}+\delta,g_m]}\{x-G(\alpha_*, x)\}}\right]+1.
$$
The circulation between those intervals cannot be more than $2k_0(\alpha_*)$ times. Let $N:=2k_0(\alpha_*) (N_1+N_2)$, then in $N$ steps the solution to \eqref{eq:PBCstoch} with $\alpha_n\equiv \alpha_*$  reaches $[K_1-\delta, K_3+\delta]$.

By Lemma \ref{lem:topor},  there exists a random moment $\mathcal N$, s.t., with probability 1, for $N$-steps in a row,  starting from~$\mathcal N$,
\[
\alpha_n=\alpha+ \ell \xi_n>\alpha+(1-\varepsilon) \ell=\alpha_*,\quad n=\mathcal N,   \mathcal N+1, \dots, \mathcal N+N.
\]
Acting as in the proof of  Lemma \ref{lem:x0beta0} and  using $x_{\mathcal N(\omega)}$ instead of $x_0$ for each $\omega\in \Omega$, we conclude that
\[
x_{\mathcal N(\omega)+n(\omega)}\in [K_{1}-\delta, \, K_{3}+\delta] \quad \mbox{for some integer} \quad n(\omega)\le 2k_0(\alpha_*), \,\, \omega\in \Omega.
\]
Since  $\alpha_n>\underline \alpha_0$ for each $n\in \mathbb N$ on all $\Omega$, and   by Lemma \ref{lem:delta01}, as soon as the solution gets into $ [K_{1}-\delta, \, K_{3}+\delta] $, it stays there and converges either to $K_3$ or to $K_1$, which concludes the proof.

\subsection{Estimation of $\underline \alpha$}
\label{subsec:exG}

In this section we discuss the case mentioned in Remark \ref{rem:exG}.
Below $ \underline \alpha_0$,  $d_i(\beta)$,  $c_i(\beta)$ ($\beta\in (0,1)$ and $i\in \mathbb N$),
$k_0(\beta)$, $\underline \alpha$ are defined as in \eqref{def:alpha0}, \eqref{def:dc4eq}, \eqref{def:0k0} and \eqref{def:underalphacd}, respectively.
Suppose  that  Assumption \ref{as:4eq1} and condition \eqref{cond:1} hold,
$x_m$ is  the only  point of maximum of $g$ on $[K_0, K_1-\delta]$,
$g$ decreases on $[K_3, g_m]$, where $g_m=g(x_m)$,  $g$ is differentiable outside of $(d_1(0), c_1(0))$ and, for some $L_1, L_2>0$,
$
g'(x)>-L_1$ for $x\in [x_m, d_1(0)]$ and $g'(x)>-L_3$ for $x\in [c_1(0), g_m]$, where, for simplicity, $L_1\ge  L_3$. Note that, since $g$ decreases on $[x_m, K_1]$, we have  $d_1(0)=g^{-1}(K_3)$ and   $c_1(0)=g^{-1}(d_1(0))$.
Also, for any $\alpha$ the function $G(\alpha, \cdot)$ has a maximum $G_{\max}(\alpha)=G(\alpha, x_{Gm})$ on $(K_0, K_1)$ and $x_{Gm}(\alpha)\in (x_m, K_1)$.

Define
\begin{equation*}
\label{def:alpha1}
\underline \alpha_1:=\max\left\{\underline \alpha_0, \,\,\frac{L_1L_3-1}{(L_1+1)(L_3+1)} \right\}.
\end{equation*}
If we prove that $\underline \alpha\le \underline \alpha_1$ then $\underline \alpha_1$ can be used for stabilization even though it might not coincide with the best (minimal) possible control.

By direct calculations we show the following.
\begin{enumerate}
\item [(i)]
$
\frac{L_1L_3-1}{(L_1+1)(L_3+1)}<\frac{L_3}{L_3+1}<\frac{L_1}{L_1+1}.
$
\item[(ii)]  If $\alpha<\frac{L_i}{L_i+1}$ then  $(1-\alpha)L_i-\alpha>0$, $i=1,3$.

\item[(iii)]   If $\alpha\in \left(\underline \alpha_1, \frac{L_3}{L_3+1}\right)$  then  $\bigl((1-\alpha)L_3-\alpha\bigr)\bigl((1-\alpha)L_1-\alpha\bigr)<1$. This holds since $z=\frac{L_1L_3-1}{(L_1+1)(L_3+1)}$ is the smallest root of the equation $\bigl((1-z)L_3-z\bigr)\bigl((1-z)L_1-z\bigr)=1$.
\end{enumerate}

By Lemma \ref{lem:comb},  for each $k\in \mathbb N$,  we have  $\hat d(0)\le d_k(0)<K_1-\delta$, $\hat c(0)\ge c_k(0)>K_1+\delta$, $g(\hat d(0))=\hat c(0)$, $g(\hat c(0))=\hat d(0)$, and, for each $\alpha\in (0,1)$, $G(\alpha, \hat d(0))<g(\hat d(0))=\hat c(0)$, $G(\alpha, \hat c(0))>g(\hat c(0))=\hat d(0)$, $d_k(0)>d_k(\alpha)$, $c_k(0)<c_k(\alpha)$. Also, $\hat d(0)>\hat d(\alpha)$ and $\hat c(0)<\hat c(\alpha)$. To prove that the last two inequalities are strict, we show that  $\{ \hat d(0),  \bar c(0) \}$ cannot be  a two-cycle for $G(\alpha, \cdot)$ with some $\bar c(0)\ge \hat c(0)$ and each $\alpha\in (0, 1)$. Indeed, assuming the contrary, we get
\[
\bar c(0)=G(\alpha, \hat d(0))=(1-\alpha)g(\hat d(0))+\alpha \hat d(0)=(1-\alpha)\hat c(0)+\alpha \hat d(0) \ge \hat c(0),
\]
which leads to a contradiction $\hat c(0)\le \hat d(0)$.

Since there is no desired stability for the original function $g$, i.e. $k_0(0)=\infty$,  we have $x_m<\hat d(0)<d_k(0)$. Also, if $k_0(\alpha)=\infty$ we have $x_m<x_{Gm}(\alpha)<\hat d(\alpha)<d_k(\alpha)$.

Suppose first that $\underline \alpha_0 <\frac{L_1L_3-1}{(L_1+1)(L_3+1)}$.
Let $\alpha_*\in  \left(\underline \alpha_1, \frac{L_3}{L_3+1}\right)$  and assume that $k_0(\alpha_*)=\infty$. Then  $x_m<x_{Gm}(\alpha_*)<\hat d(\alpha_*)<d_k(\alpha_*)<d_1(0)$, $c_1(0)<\hat c(0)$.
If we show that $\{ \hat d(\alpha_*), \hat c(\alpha_*) \}$ is not a two-cycle for $G(\alpha_*, \cdot)$, it would contradict to the assumption $k_0(\alpha_*)=\infty$ and prove that $\underline \alpha>\underline \alpha_1$.

Let $x\in  [x_m, d_1(0)]$ be such that   $G(\alpha_*, x) \in [c_1(0), g_m]$, then
\begin{equation}
\label{est:c0}
\begin{split}
&G(\alpha_*, x)-\hat c(0)=G(\alpha_*, x)-g(\hat d(0))<G(\alpha_*, x)-G(\alpha_*,\hat d(0))\\ = & (1-\alpha_*)[g(x)-g(\hat d(0))]+\alpha_*(x-\hat d(0))
\le [(1-\alpha_*)L_1-\alpha_*](\hat d(0)-x).
\end{split}
\end{equation}
Since $c_1(0)<\hat c(0)<g_m$,
\begin{align*}
\hat d(0)-G^2(\alpha_*, x)= & g(\hat c(0))-G^2(\alpha_*, x)\le G(\alpha_*,\hat c(0))-G^2(\alpha_*, x)
\\ = & (1-\alpha_*)[g(\hat c(0))-g(G(\alpha_*, x))]
+\alpha_*(\hat c(0)-G(\alpha_*, x)) \\ \le &  [(1-\alpha_*)L_3-\alpha_*](G(\alpha_*, x)-\hat c(0)).
\end{align*}
Since  $\alpha_*<\frac{L_3}{L_3+1}$, by  (ii)-(iii) we have $(1-\alpha_*)L_i-\alpha_*>0$, $i=1,2$, $\prod_{i=1}^2 [(1-\alpha_*)L_i-\alpha_*]<1$ and therefore
\begin{align*}
&\hat d(0)-G^2(\alpha_*, x)\le  [(1-\alpha_*)L_3-\alpha_*] [(1-\alpha_*)L_1-\alpha_*](\hat d(0)-x)<\hat d(0)-x,
\end{align*}
which implies $G^2(\alpha_*, x)>x$.  If $ \{ \hat d(\alpha_*), \hat c(\alpha_*) \}$ is a two-cycle for $G(\alpha_*, \cdot)$ then  $G(\alpha_*, \hat d(\alpha_*))=\hat c(\alpha_*)\in [c_1(0), g_m]$, $G(\alpha_*, \hat c(\alpha_*))=\hat d(\alpha_*)\in [x_m, d_1(0)]$ and
\[
G^2(\alpha_*, \hat d(\alpha_*))=G(\alpha_*, G(\alpha_*, \hat d(\alpha_*)))=G(\alpha_*, \hat c(\alpha_*))=\hat d(\alpha_*).
\]
However, from the above, we should have $G^2(\alpha_*, \hat d(\alpha_*))>\hat d(\alpha_*)$,  which is  a contradiction.

If $\alpha_*>\frac{L_1}{L_1+1}$ then $(1-\alpha_*)L_1-\alpha_*<0$ and from \eqref{est:c0} we conclude that $G(\alpha_*, x)<\hat c(0)$, which implies that for some $k\in \mathbb N$, we have $G(\alpha_*, x)<c_k(0)<c_k(\alpha_*)\le \hat c(\alpha_*)$, so $\alpha_*>\underline \alpha$, see \eqref{def:underalphacd}. If $G(\alpha_*, x) \in [\hat c(0), g_m]$ and $\frac{L_1}{L_1+1}>\alpha_*>\frac{L_3}{L_3+1}$, we get  $G^2(\alpha_*, x)>\hat d(0)\ge x$. From here we proceed as above  and get a contradiction for $x= \hat d(\alpha_*)$.

Assume now that $\underline \alpha_0 \ge \frac{L_1L_3-1}{(L_1+1)(L_3+1)}$. The cases $\alpha_*>\frac{L_1}{L_1+1}$ and $\frac{L_1}{L_1+1}>\alpha_*>\frac{L_3}{L_3+1}$ were considered above. If $\frac{L_3}{L_3+1}>\underline \alpha_0$ then the case
$\frac{L_3}{L_3+1}>\alpha_*>\underline \alpha_0$ has already been discussed.

\subsection{Control for an arbitrary  number of equilibrium points}
\label{subsec:arbap}
\subsubsection{Classification of intervals}
\label{subsubsec:classif}

We set, for each $i\in I^+$ and $i_0$ defined in \eqref{def:i0},
\begin{equation}
\label{def:li-+}
l_{+-}(i):=\left \{
  \begin{array}{ll}
\inf\{j>i:j\in I^-  \},& \mbox{if} \quad  \{j>i:j\in I^-  \}\neq \emptyset,\\
  i_0, &  \mbox{\rm otherwise} ,
  \end{array}
\right.
\end{equation}
and  for each $i\in I^-$
\begin{equation}
\label{def:li+-}
l_{-+}(i):=\left \{
  \begin{array}{ll}
\inf\{j>i:j\in I^+ \},& \mbox{if} \quad  \{j>i:j\in I^+ \}\neq \emptyset,\\
  i_0, &  \mbox{\rm otherwise}.
  \end{array}
\right.
\end{equation}

Introduce  the set of all consecutive intervals  $
{\mathcal V}:=\{(K_p, K_{p+1}),  K_p, K_{p+1}\in \mathbb K,\,  p=0, 1, \dots, j_0-1\}
$
with the ends in the set of equilibrium points $ \mathbb K$, see \eqref{eq:equil1}.
 We create two stages of control: with $\alpha^{(1)}$ and $\alpha^{(2)}$, $1>\alpha^{(2)}\ge \alpha^{(1)}>\bar L/(\bar L+1)$. We also distinguish between two types of  blocks of consecutive intervals from $\mathcal V$.  After application of the first stage of control,   $\alpha_n\in \left(\alpha^{(1)}, \, 1\right)$,  there will be no communication between any  blocks.  After application of the second stage of control,   $\alpha_n\in \left(\alpha^{(2)}, 1\right)$,  there will be no infinite circulation inside of any block.

There are two blocks of the first type,  $\mathcal V_0$ and $\tilde {\mathcal V}$. $\mathcal V_0$ consists of the  initial intervals from $\mathcal V$ with indexes $p$ from $0$ to some $m_0$,  $\tilde{\mathcal V}$ is
made of  the   intervals with indexes $p$ from some $\tilde m$ to $ j_0-1$.
After the application of the first control stage, if a solution gets into either  $\mathcal V_0$  or $\tilde {\mathcal V}$, it stays in the block and converges to one of equilibrium points.
There are $s_0$ blocks in the second group, and each  block of the second type contains two groups of intervals, between which   circulation is possible. In other words,
 \begin{equation}
 \label{def:mathcalV}
 {\mathcal V}= \mathcal V_{0}\bigcup \left[\bigcup_{s=0}^{s_0-1} \mathcal V_{2s+1}\right]\bigcup \tilde {\mathcal V}.
\end{equation}
Note that  $\mathcal V_0$ or  $\tilde{\mathcal V}$,  or all of $ \mathcal V_{2s+1}$ can be empty.

Now we proceed to a detailed definition  of blocks.
Assume $0\in I^-$, i.e. the left Lipschitz constant at $K_1$ is finite and less than the right Lipschitz constant at $K_1$, see \eqref{def:L2i+1} and \eqref{def:I12}. Denote $$\mathcal V_0:=\{(K_{p}, K_{p+1}), 0\le p\le 2m_0-1\},~~ m_0:=l_{-+}(0)\in I^+,$$  see 
\eqref{def:li+-}, so $K_{2m_0+1}$ is the first from the left equilibrium where the right Lipschitz constant  is finite and less than the left Lipschitz constant.

By Lemma \ref {lem:comb}, a solution remains in any interval  $[K_{2j}, K_{2j+1}]\in \mathcal V_0$ forever. If $ x_0\in [K_{2j-1}, K_{2j}]$, $j\le m_0$, then a solution can get out of that interval but only to the left and then remains in one of $[K_{2i}, K_{2i+1}]$, $0\le i\le j-1$.  So circulation is impossible if $ x_0$ is from one of the intervals of $\mathcal V_0$.
Note that if $0\in I^+$,  the set $\mathcal V_0$ is empty and $m_0=0$.
	
 Denote now  $\mathcal K_{0, 1}:=\left\{[K_{p}, \, K_{p+1}], \, 2m_0\le p\le 2m_1-1\right\}$, where $m_1=l_{+-}(m_0)\in I^-$, see  \eqref{def:li+-}, so $K_{2m_1+1}$ is the first equilibrium after $K_{2m_0}$ where the left Lipschitz constant  is finite and less than the right Lipschitz constant.  From any  interval   $(K_{2j}, \, K_{2j+1})\in \mathcal K_{0, 1}$, a solution can move to the right.  If it gets into $[K_{2j+1}, \, K_{2j+2}] \in \mathcal K_{0, 1}$,  it stays there. The first case of circulation is possible when a solution jumps over $K_{2m_1}$ to the interval of type $[K_{2j+1}, \, K_{2j+2}]$ from  $\mathcal K_{1, 2}:=\left\{(K_{p}, \, K_{p+1}), \, 2m_1\le p\le 2m_2-1 \right\}$, $m_2:=l_{-+}(m_1)\in I^+$. Infinite circulation happens  when a solution attends  infinitely many times some  intervals of type $(K_{2j}, \, K_{2j+1})$, in ascending order, from $\mathcal K_{0,1}$, and some  intervals of type $(K_{2k+1}, \, K_{2k+2})$,  in descending order,  from $\mathcal K_{1, 2}$, where $j\le m_1-1$ and $k\ge m_1$, since,  by Lemma \ref{lem:comb}, if a solution gets into  the interval $(K_{2m_1-1}, \, K_{2m_1+1})$, it stays in it.

We denote  $\mathcal V_1:=\mathcal K_{0, 1}\cup \mathcal K_{1, 2}=\left\{[K_{p}, \, K_{p+1}], \, 2m_0\le p\le 2m_2-1\right\}$.
Inductively, we
define  $\mathcal K_{2s, 2s+1}:=\left\{(K_{p}, \, K_{p+1}), \, 2m_{2s}\le p\le 2m_{2s+1}-1 \right\}$ with $m_{2s+1}:=l_{+-}(m_{2s}) \in I^-$,
 and  $\mathcal K_{2s+1, 2s+2}:=\left\{(K_{p}, \, K_{p+1}), \, 2m_{2s+1}\le p\le 2m_{2s+2}-1 \right\}$ with $m_{2s+2}:=l_{-+}(m_{2s+1})\in
I^+$,
 \begin{equation}
 \label{def:mathcalVs}
\begin{array}{l}
\displaystyle  \mathcal V_{2s+1}:=\mathcal K_{2s, 2s+1}\cup \mathcal K_{2s+1, 2s+2}=\left\{[K_{p}, \,
K_{p+1}], \, 2m_{2s}\le p\le 2m_{2s+2}-1\right\}, \\ s\le s_0-1,
\end{array}
 \end{equation}
 where $s_0=\max\{s:2m_{2s}\le j_0  \}$.  Similarly, infinite circulation happens  when a solution attends  infinitely many times some  intervals of type $(K_{2j}, \, K_{2j+1})$, in ascending order, from $\mathcal K_{2s, 2s+1}$, and some  intervals of type $(K_{2k+1}, \, K_{2k+2})$,  in descending order,  from $\mathcal K_{2s+1, 2s+2}$, where $j\le m_s-1$ and $k\ge m_s$, since,  by Lemma \ref{lem:comb}, if a solution gets into  the interval $(K_{2m_s-1}, \, K_{2m_s+1})$ it stays in it.

To illustrate this, we use again the fourth iterate of Ricker's map, see Fig.~\ref{fig:Rick4}
and Section~\ref{sec:arbeq}, and have $m_j=j$, $j=0,\dots, 4$, $\mathcal K_{0, 1}=\{(K_0, K_1), \, (K_1, K_2)\}$, $\mathcal K_{1, 2}=\{(K_2, K_3), \, (K_3, K_4)\}$, $\mathcal K_{2, 3}=\{(K_4, K_5), \, (K_5, K_6)\}$, $\mathcal K_{3, 4}=\{(K_6, K_7), \, (K_7, \infty)\}$.

If $2m_{2s_0} = j_0$ then  $\tilde{\mathcal V}=\emptyset$.
 If    $2m_{2s_0} < j_0$  then there is $i^*$ s.t. either $2m_{2s_0+i^*} =j_0$ or $2m_{2s_0+i^*} +1=j_0$. In both cases we have $m_{2s_0+i}\in I^+$  for all $i=0, 1, \dots, i^*$, see \eqref{def:I12}, and we can define non-empty
 \[
\tilde {\mathcal V}:=\left\{[K_{p}, \, K_{p+1}], \, 2m_{2s_0}\le p\le j_0-1\right\}.
\]
 Note that $j\in I^+$ in each interval $(K_{2j}, K_{2j+1})$ and $(K_{2j+1}, K_{2j+2})$ from $\tilde {\mathcal V}$. Thus, from each $[K_{2j}, K_{2j+1}]\in \tilde {\mathcal V}$, a solution can move only to the right and not further than $K_{j_0}$, and it remains in each $[K_{2j+1}, K_{2j+2}]\in \tilde {\mathcal V}$.

 All the above gives us the decomposition in \eqref{def:mathcalV}. Therefore, for each function $g$ satisfying Assumption \ref{as:LL1},
 the interval  $[K_0, K_{j_0}]$  can be represented as a union of blocks \eqref{def:mathcalV}. The only blocks where a cycle can occur are blocks $\mathcal V_{2s+1}$ with  $0\le s\le s_0-1$.

 In Sections \ref{subsec:det}-\ref{subsec:stochgen},  for a block  of this type  we briefly  discuss how to introduce
a new control  $\beta_{1}\in \left(\bar L/(\bar L+1), \,1\right)$, which eliminates  the circulation.


\subsubsection{Deterministic control}
\label{subsec:det}

Consider one of the blocks $\mathcal V_{2s+1}$ defined by \eqref{def:mathcalVs} and  consisting of two adjacent groups $\mathcal K_{2s, 2s+1}$ and $\mathcal K_{2s+1, 2s+2}$ with the lengths of $m$ and $r$, respectively.
To simplify the description of the structure for $\mathcal V_{2s+1}$, we shift indexes to zero, i.e. $s=0$. We  assume that the block ${\mathcal V}_1$ contains equilibrium points
 \begin{equation}
\label{def:Kmr1}
\begin{array}{l}
\displaystyle \mathbb K_1=\left\{K_{0},  K_{1}, \dots K_{2m-1}, K_{2m}, K_{2m+1}, K_{2m+2}, \dots,  K_{2(m+r)}\right\}\in \mathbb K,
\\
m\in \mathbb N_0, \,  r\in \mathbb N, \, r\ge 2 
\end{array}
\end{equation}
and satisfies
\begin{assumption}
\label{as:Kmr}
In the set of equilibrium points defined by \eqref{def:Kmr1}, $0, 1, 2, \dots, m \in I^+$  
 and also $m+1, \dots m+r\in I^-$ if $K_{2(m+r)}<\infty$, while $m+1, \dots m+r-1\in I^-$ if $K_{2(m+r)}=\infty$, where  $I^+$ and $I^-$ are defined as in \eqref{def:I12}.
\end{assumption}
Analogously to Theorem~\ref{thm:stoch2sides}, we can formulate and prove the result about stability on the block $\mathcal V_1$, see Lemma~\ref{lem:detmr}.

Assumption \ref{as:Kmr} implies that $m+1=l_{+-}(0)\in I^-$, i.e. the left Lipschitz constants at all $K_{2j+1}$, $j\le m$,  are finite and less than the right Lipschitz constants,  see definitions \eqref{def:li+-}, and  for  $j>m$ the opposite is valid, so
$\mathcal K_{0,1}=\{(K_p, K_{p+1}), \, \, 0\le p\le 2m\}$, $\mathcal K_{1,2}=\{(K_p, K_{p+1}), \, \, 2m\le p\le 2(m+r)-1\}$.
We define
\begin{equation}
\label{def:beta0mr}
\beta_{0}  :=  \inf\left\{\beta \in \left(\frac{\bar L}{\bar L+1},1 \right)\!: \!\!\!\!\inf_{x\in (K_{1}, K_{2(m+r)})}\!\! \!\!\!\! \!\!\!\! \!\! G(\beta, x)>K_{0} , \!\!\!\! \!\!\!\! \!\!
\sup_{x\in \left(K_{0}, K_{2(m+r)-1}\right)}\!\!\!\!\!\! \!\!\!\! \!\! G(\beta, x)<K_{2(m+r)} \right\},
\end{equation}
where for $K_{2(m+r)}=\infty$ the second inequality in the right-hand side of \eqref{def:beta0mr} holds unconditionally.

Fix some  $\beta\in \left(\beta_{0}, \, 1\right)$.
Definition \eqref{def:beta0mr} guarantees $$G(\beta, \cdot):(K_{0}, K_{2(m+r)})\to (K_{0}, K_{2(m+r)}).$$ Also,  applying Lemma~\ref{lem:comb} we conclude that $$G(\beta, \cdot):(K_{2m+1}, K_{2m+3})\to (K_{2m+1}, K_{2m+3})$$ and that  any cycle is possible only outside of $(K_{2m+1},K_{2m+3})$. Following the procedure introduced in Section~\ref{subsec:trap2},
  we extend the interval $(K_{2m+1}, K_{2m+3})$ to keep this property and then introduce the smallest $\beta$, for which infinite circulation of a solution between $\mathcal K_{0,1}$ and $\mathcal K_{1,2}$ becomes impossible.

Reasoning as in Section \ref{subsec:trap1}, we assume that $\max_{y\in [K_0, K_{2m+1}] }g(x) =: g_m
>K_{2m+3}$ and consider only such $\beta$ that $\max_{y\in [K_0, K_{2m+1}] }G(\beta, x)>K_{2m+3}$.
Define now  $\underline\kappa(\beta)$, the largest point of maximum  of $G(\beta, \cdot)$ on $(K_{0}, \, K_{2m+1})$, and $\bar \kappa(\beta)$,  the smallest point of minimum of $G(\beta, \cdot)$ on $(K_{2m+3}, \, g_m)$:
$\underline\kappa(\beta)=\sup\{y\in  \left(K_{0}, \, K_{2m+1}\right): G(\beta, y)=\sup_{x\in (K_{0}, K_{2m+1})}G(\beta, x) \},$
$\bar \kappa(\beta)= \inf\{y\in  \left(K_{2m+3}, \, g_m\right):   G(\beta, y)=\inf_{x\in (K_{2m+3}, g_m)}G(\beta, x) \}$.
By  \eqref{def:dc4eq} we introduce two convergent sequences of points, $(d_n(\beta))_{n\in \mathbb N}$ and $(c_n(\beta))_{n\in \mathbb N}$, located in $(K_0, K_{2m+1})$ and in $(K_{2m+3}, \, g_m)$, respectively, only with  $d_0:=K_{2m+1}$, $c_0:=K_{2m+3}$.
Now, as in Section \ref{subsec:trap2}, we define $k_0(\beta)$, $\hat c(\beta)$, $\hat d(\beta)$ by  \eqref{def:0k0}, \eqref{def:limdc}.
Analogues of Lemmata \ref{lem:propseqcd} and \ref{lem:x0beta0} hold when Assumption \ref{as:4eq1} is substituted by Assumptions~\ref{as:LL1},\ref{as:LLj0},\ref{as:Kmr}.
The interval $(\hat d(\beta), \hat c(\beta))$ includes $(K_{2m+1}, K_{2m+3})$ and is  invariant under $G(\beta, \cdot)$. Note  that $k_0(\beta)=\infty$ if and only if $\{ \hat c(\beta), \hat d(\beta) \}$ is a two-cycle  for $G(\beta, \cdot)$, so for this particular $\beta$ the interval of the initial values  with the desired convergence cannot be increased.
Moreover, the bound for control $\alpha$ is sharp: if it is smaller, a cycle rather than an equilibrium can be an attractor.
We introduce $\beta_1$ as
\begin{equation}
\label{def:beta1mr}
\beta_{1}:=\inf \left\{ \beta \in (\beta_{0},1):
\!\!\!\!\!\!\max_{x\in [K_{0}, K_{2m+1}]}
\!\!\!\!\!\! G(\beta, x) <  \hat c(\beta) ~ \mbox{or} \!\!\!\!\inf_{y\in (K_{2m+3}, g_m) }
\!\!\!\!\!\!G(\beta, y) > \hat d(\beta) \right\},
\end{equation}
which is larger than $\beta_0$, and show that $\beta_1$ is well-defined.
The proof of the following lemma, which is the main result of this section concerning stability on the block $\mathcal V_1$ is similar to the proof of Theorem  \ref{thm:detcd}.
\begin{lemma}
\label{lem:detmr}
Let Assumptions \ref{as:LL1}, \ref{as:LLj0} and \ref{as:Kmr} hold, and $I$, $\beta_{1}$  be defined as in \eqref{def:i0} and \eqref{def:beta1mr}.
Let $ \alpha_*\in(\beta_{1}, 1)$, $\alpha^*\in (\alpha_*, 1)$, $\alpha_n\in [\alpha_*, \alpha^*]$ for all $n\in \mathbb N$, and $x$ be a solution  to \eqref{eq:PBCvar} with $x_0\in \left(K_0, K_{2(m+r)}\right)$. Then $x$ converges  either to 
some  $K_{2i+1}$, $i\in I$,
or to an equilibrium inside $(K_p, K_{p+1})$, $0\le p\le j_0-1$.
\end{lemma}

\begin{remark}
\label{rem:2cycle}
Reasoning as in  Lemmata  \ref{lem:propseqcd} and \ref{lem:k0}, we can show that  the choice of lower bound $\beta_1$ by \eqref{def:beta1mr} is quite sharp: when
$\beta_1>\beta_0$, existence of a two-cycle for some $\beta>\beta_0$ confirms that not all solutions to \eqref{eq:PBCvar} converge to an equilibrium.
 \end{remark}

   \subsection{Proof of Theorem \ref{thm:det}}

According to Section \ref{subsubsec:classif}, all the intervals $(K_0, K_{j_0})$ can be represented by \eqref{def:mathcalV} as a union of blocks of the intervals $$\mathcal V_{0}\bigcup \left[\bigcup_{s=0}^{s_0-1} \mathcal V_{2s+1}\right]\bigcup \tilde {\mathcal V},$$  where inside of blocks $\mathcal V_{0}$ and $\tilde{\mathcal V}$ a solution cannot circulate.
The first stage of control, $\alpha^{(1)}\in (0,1)$, stops communication between blocks. By Assumptions \ref{as:LL1}-\ref{as:LLj0} and the form of the function $G(\alpha, x)$, see \eqref{def:G}, such $\alpha^{(1)}$ exists.  Applying  Lemma \ref{lem:detmr}, for each block of type $\mathcal V_{2s+1}$ we find a control $\beta_1(s)$ which stops circulation inside of $\mathcal V_{2s+1}$.  By setting $\alpha_*>\max\{\alpha^{(1)},  \beta_1(s), s=0, \dots, s_0-1\}$ we conclude the proof.

\subsubsection{Stochastic control}
\label{subsec:stochgen}
In this section we briefly discuss some situations when a control perturbed by  an additive noise, $\alpha+\ell \xi_{n}$,  can improve the deterministic result, decreasing the lower bound of $\alpha$.

%
\begin{assumption}
\label{as:stochmr}
Let in Assumptions~\ref{as:LL1},\ref{as:LLj0},\ref{as:Kmr} the set $\mathbb K$ be defined by  \eqref{def:Kmr1}, and $g(x)\neq x$ for $x\in (K_p, K_{p+1})$ and all $K_p, K_{p+1}\in \mathbb K$.
\end{assumption}
Let $\beta_0$, $\underline \kappa(\beta)$, $d_k$, $\bar\kappa(\beta)$, $c_k$ be as in Section \ref{subsec:det} (see also  Section \ref{subsec:trap2}).
Circulation of a solution is possible only if both  $G(\beta_0, \underline \kappa(\beta_0))>K_{2m+3}$ and $G(\beta_0, \bar \kappa(\beta_0))< K_{2m+1}$ hold. In this case we can define
\begin{equation}
\label{def:tautheta}
\begin{split}
\tau:=\inf\{s\in\{0, \dots, m\}:K_{2s+1}\ge  \underline \kappa(\beta_0)  \},\,\, & \mbox{where} \,\,  \underline \kappa(\beta_0)\in (K_{2\tau}, K_{2\tau+1}),\\
\theta:=\sup\{s>m+1:K_{2s+1}\le  \bar\kappa(\beta_0)  \},\,\, & \mbox{where} \,\,  \bar \kappa(\beta_0)\in (K_{2\theta+1}, K_{2\theta+2}).
\end{split}
\end{equation}
Notation \eqref{def:tautheta} yields that $G(\beta_0, \cdot)$ takes its maximum with respect to all intervals to the left of $ K_{2\tau}$  on  $ [K_{2\tau}, K_{2\tau+1}]$.
Applying Lemma \ref{lem:Gv}(iv), we can show that the same holds for  $G(\beta, \cdot)$, $\beta \in (\beta_0,1)$. Similarly, $G(\beta_0, \cdot)$ takes its minimum  with respect to all the intervals to the right of
$K_{2\theta+2}$  on  $ [K_{2\theta +1}, K_{2\theta+2}]$,  and the same holds for  $G(\beta, \cdot)$, $\beta \in (\beta_0,1)$.  Define now
\begin{equation}
\label{def:beta21}
\begin{split}
\beta_{21}=\inf
& \left\{\beta  \in (\beta_0,1) :  \!\! \!\! \sup_{x\in [K_{2\tau+1},K_{2m+1}]}\!\! \!\! \!\! G(\beta, x) < K_{2m+3}, \right.
\\
& \left. \sup_{x\in [K_{2\tau}, K_{2\tau+1}]}\!\! \!\! \!\! G(\beta, x) < K_{2\theta+2} \right\},
\\
\beta_{22}=\sup & \left\{ \beta \in (\beta_{21},1):\sup_{x\in [K_{2\tau}, K_{2\tau+1}]}G(\beta, x) > K_{2\theta+1}
\right\},\\
\beta_{23}=\inf & \left\{ \beta \in (\beta_{0},1) : \!\!\!\!\inf_{x\in [K_{2m+3}, \, K_{2\theta+1}]}
\!\!\!\!
G(\beta, x) > K_{2\tau+1},\right. \\
& \left. \inf_{x\in [ K_{2\theta+1}, \,  K_{2\theta+2}]}\!\!\!\! \!\! G(\beta, x) > K_{2\tau} \right\}, \\
\beta_{24}=\sup & \left\{ \beta \in (\beta_{23},1):\inf_{x\in [ K_{2\theta+1}, \,  K_{2\theta+2}]}G(\beta, x) < K_{2\tau+1} \right\}.
\end{split}
\end{equation}
\begin{remark}
\label{rem:beta1234}
If the set on the third line of \eqref{def:beta21} is empty then $\beta_{21}$  is the stabilization lower bound. If the set on the last line is empty, we  get the stabilization lower bound $\beta_{23}$. Since in these two cases a stabilization bound is known, we proceed to establishing the required $\beta$ when these sets are non-empty.

For $\beta\in (\beta_{21}, \beta_{22})$ we have $\sup_{x\in [K_{2\tau}, K_{2\tau+1}]}G(\beta, x) \in( K_{2\theta+1}, K_{2\theta+2})$, and for $\beta\in (\beta_{23}, \beta_{24})$ we have $\inf_{x\in [K_{2\theta+1}, \,  K_{2\theta+2}]}G(\beta, x) \in( K_{2\tau}, K_{2\tau+1})$.

If $\beta_{22}\le \beta_{23}$  then $\sup_{x\in [K_{2\tau}, K_{2\tau+1}]}G(\beta, x) \le  K_{2\theta+1}$
for any  $\beta>\beta_{22}$, which means that the circulation of a solution between the intervals $ (K_{2\tau}, K_{2\tau+1})$ and $(K_{2\theta+1}, \,  K_{2\theta+2})$ is impossible. Therefore we consider only the case
when $\beta_{22}>\beta_{23}$.
\end{remark}

\begin{assumption}
\label{as:bet2224}
Assume that $\beta_{2j}$, $j=1,\dots, 4$,  be well defined by \eqref{def:beta21} and
\begin{equation}
\label{cond:notempty}
(\beta_{21}, \beta_{22})\cap (\beta_{23}, \beta_{24})=({\underline\beta}_2, \bar\beta_2)\not = \emptyset.
\end{equation}
\end{assumption}

\begin{remark}
\label{rem:aux3}
For  $\beta\in ({\underline\beta}_2, \bar\beta_2)$, Assumption \ref{as:bet2224} implies
\begin{equation}
\label{eq:incl}
\begin{split}
& \underline\kappa(\beta), d_k(\beta), \hat d(\beta)\in (K_{2\tau}, K_{2\tau+1}) \mbox{~~~and~} \\
& \bar\kappa(\beta), c_k(\beta),  \hat c(\beta)\in (K_{2\theta+1}, K_{2\theta+2}),~ k\in \mathbb N.
\end{split}
\end{equation}
\end{remark}
\begin{remark}
\label{rem:cases}
Note that Assumption \ref {as:bet2224} means that there is a possibility of circulation of a solution between intervals $( K_{2\tau}, K_{2\tau+1})$ and $( K_{2\theta+1}, K_{2\theta+2})$. By the choice of $\beta_{ij}$ in \eqref{def:beta21}, we eliminate circulation inside of  $(K_{2\tau+1}, \, K_{2\theta+1})$, turning $( K_{2\tau+1}, K_{2\theta+1})$ into a new trap.  By \eqref{def:beta21} and \eqref {cond:notempty}
the  solution cannot get above $ K_{2\theta+2}$ from $( K_0, K_{2\tau})$, and it  cannot get below $K_{2\tau}$ from $( K_{2\theta+1}, K_{2(m+r)})$.
\end{remark}

\begin{remark}
\label{rem:stoch}
If   $\max_{x\in [K_{2m}, K_{2m+1}]}g(x)\le K_{2m+4}$  and  $\min_{x\in [K_{2m+3},
K_{2m+4}]}g(x)\ge K_{2m}$ then $\beta_{21}=\beta_{23}=\beta_0$, and if $\min\{\beta_{22}, \beta_{24}\}>\beta_0$, circulation between intervals
is possible.
Such situation is illustrated in  Section~\ref{sec:ex}, Examples~\ref{ex:Ricker2} and \ref{ex:infder}, where $m=0$ and $r=2$.
\end{remark}

To show that under Assumptions \ref{as:stochmr}  and \ref{as:bet2224} introduction of the noise into control can improve the deterministic result, we set
\begin{equation}
\label{def:beta3mr}
\begin{split}
\beta_{3}:= \inf & \left\{\beta\in ({\underline\beta}_2, \bar\beta_2): \!\!\!\! \sup_{x\in  (K_{2\tau}, K_{2\tau+1})}\!\!\!\! G(\beta, x) <  \hat c(\beta) \, \mbox{~~~or} \right.
\\
& \left. \inf_{y\in(K_{2\theta+1}, K_{2\theta+2})} \!\!\!\! G(\beta, y) > \hat d(\beta)\right\}.
\end{split}
\end{equation}
Relations \eqref{def:beta21}, \eqref{cond:notempty} and \eqref{eq:incl} imply that $\beta_3$ is well defined.

\begin{remark}
\label{rem:beta23}
If $\underline \beta_{2}=\beta_{3}$ and $\alpha\in (\underline \beta_{2}, \bar\beta_2)$, by Lemma \ref{lem:detmr},  we do not need to introduce
a noise perturbation $\ell \xi_n$ to achieve stabilization of all equilibria in $\mathbb K$.
Any small noise perturbation with  $\ell<
\min\{\bar\beta_{2} -\alpha, \,\alpha-\underline \beta_{2}\}$ keeps this stability which was achieved by the deterministic part  $\alpha$ of the  control $\alpha_n=\alpha+\ell \xi_n$. Theorem \ref{thm:stochmr} below is devoted to the case when $\underline \beta_{2}<\beta_{3}$, and the noise plays an active role in stabilization of the equilibrium points $K_{2\tau+1}$ and $K_{2\theta+1}$.
 \end{remark}
 Choose some
\begin{equation}
\label{def:almr0}
\alpha\in\left((\underline \beta_{2}+\beta_{3})/2, \, \bar \beta_2\right),
\quad \ell\in \left( \beta_{3}-\alpha,\, \min\{\alpha-\underline \beta_{2}, \, \bar \beta_2-\alpha\} \right).
\end{equation}
By the definition of $\beta_3$ in \eqref{def:beta3mr} and the choice of $\alpha$ in \eqref{def:almr0},  the second interval in \eqref{def:almr0} is non-empty.

\begin{theorem}
\label{thm:stochmr}
Let Assumptions \ref{as:LL1}-\ref{as:bet2224} hold.  Let  $\mathbb K$ be defined as in \eqref{def:Kmr1} and $\alpha$, $\ell$ satisfy
\eqref{def:almr0}.  Then any solution  to \eqref{eq:PBCstoch}  with $x_0\in (K_{0}, K_{2(m+r)})$ converges to one of the equilibrium points in $\mathbb K$, with a total probability of one.
\end{theorem}
\begin{proof}



 Set
$\tilde G(\beta_0):=\max_{x\in [K_{0}, K_{2m+1}]} G(\beta_0, x) $ for $\beta_0$ as in \eqref{def:beta0mr}. Choose $\alpha$ and $\ell$ satisfying  \eqref{def:almr0} arbitrarily. Since $\alpha+\ell\in\left(\beta_{3}, \bar \beta_2\right)$,  we can find an $\varepsilon>0$ such that
$\alpha+(1-\varepsilon)\ell\in \left(\beta_{3}, \, \bar \beta_2\right)$, and also we have $\alpha-\ell>\underline \beta_2$.
Relations  \eqref{def:tautheta}, \eqref{def:beta21}, \eqref{cond:notempty},  \eqref{eq:incl}
imply that whenever $\beta\in ({\underline\beta}_2, \bar\beta_2)$, we have
$
 \inf_{x\in (K_{2m+3}, K_{2(m+r)})}G(\beta, x)>K_{2\tau}, \quad \sup_{x\in (K_{0},K_{2m+1})}G(\beta, x)<\min\{\tilde G(\beta_0), K_{2\theta+2}\}.
$
Therefore  there exists $\sigma=\sigma(\beta)>0$ s.t.
 \begin{equation}
\label{def:sigma}
\begin{split}
& \inf_{x\in (K_{2m+3}, K_{2(m+r)})}\!\!\!\! G(\beta, x)>K_{2\tau}+\sigma, \\
& \sup_{x\in (K_{0},K_{2m+1})}
\!\! G(\beta, x)<\min\{\tilde G(\beta_0), K_{2\theta+2}\}-\sigma.
\end{split}
 \end{equation}
 By Remark \ref{rem:beta1234}  we have $\min\{\tilde G(\beta_0), K_{2\theta+2}\}>K_{2\theta+1}$.  If \eqref{def:sigma} holds for a given $\beta_1$ and $\sigma$, it is also satisfied for the same $\sigma$ and any $\beta \in (\beta_1,1)$. So the  inequality \eqref{def:sigma}  holds for any
 \begin{equation}
\label{def:sigma1}
 \sigma\le \left\{\sigma(\alpha-\ell), \,\, \frac{K_{2\tau+1}-K_{2\tau}}2, \,\, \frac{\min\{\tilde G(\beta_0), K_{2\theta+2}\}-K_{2\theta+1}}2\right\}.
 \end{equation}
Define, for  $\beta\in [\alpha-\ell, \alpha+\ell]$,
  \begin{equation}
\label{def:deltabeta}
\delta(\beta):=\min\left\{K_{2\tau+1}-d_1(\beta), \,\,  c_1(\beta) - K_{2\theta+1}\right\}.
\end{equation}
By \eqref{eq:incl} the right-hand side in \eqref{def:deltabeta} is positive, and, by Lemma~\ref{lem:propseqcd} (vi),  we have  $\delta(\alpha+\ell)\ge \delta(\beta)\ge\delta(\alpha-\ell)$.
Define, for $\sigma$ satisfying \eqref{def:sigma1} and $\delta\le \delta(\alpha-\ell)$, which is small enough,
 \begin{equation*}
\label{def:DN12}
\begin{split}
 &\Delta_1:=\!\! \min\left\{G(\alpha+\ell, x)-x, \,x\in [K_{2\tau}+\sigma, K_{2\tau+1}-\delta]\right\}, \\
 &\Delta_2 := \!\!\min\!\left\{x-G(\alpha+\ell, x), \, x\in \left[K_{2\theta+1}+\delta, \min\{\tilde G(\beta_0), K_{2\theta+2}\}-\sigma\right]\right\}, \\
 & N_1:=\left[\frac {K_{2\tau+1}-\delta-K_{2\tau}-\sigma}{\Delta_1}\right]+1, \\ & N_2:=\!\!\left[\frac {\min\{\tilde G(\beta_0), K_{2\theta+2}\}-\sigma-K_{2\theta+1}-\delta}{\Delta_2}\right]+1,
 \end{split}
\end{equation*}
so the solution $x_n$ starting in each interval gets out of it  in less than $N_1$ (respectively, $N_2$) steps, for each $n\in \mathbb N$.
Now we follow the steps of the proof of Lemmata~\ref{lem:x0beta0} and ~\ref{lem:detmr}. We put
\begin{equation*}
\label{def:alpha*1}
\alpha_*:=\alpha+\ell (1-\varepsilon), \quad \alpha^*:=\alpha+\ell, \quad \mbox{which implies} \quad \underline \beta_2<\beta_3<\alpha_*<\alpha^*<\bar \beta_2.
\end{equation*}
Define $k_0=k_0(\alpha_*)$ by \eqref{def:0k0} and note that, since  $\alpha_*>\beta_3$ and by definition \eqref{def:beta3mr} of $\beta_3$  (see also \eqref{def:dc4eq}), we have   $k_0(\alpha_*)<\infty$.  Assume that  $k_0\ge 2$.   Let  $N:= k_0(\alpha_*) (N_1+N_2)$.
By Assumption \ref{as:noise}, we have $\mathbb P\{\alpha+\ell \xi_n> \alpha+\ell(1-\varepsilon)= \alpha_*\}>0$.   Applying Lemma \ref{lem:topor}, we conclude  that there exists a   random moment $\mathcal N$, s.t., with probability 1, for $N$ steps in a row,  starting from~$\mathcal N$,
\[
\alpha_n=\alpha+ \ell \xi_n> \alpha_*>\beta_{3},\quad n\in U_\mathcal N:=\{\mathcal N,   \mathcal N+1, \dots, \mathcal N+N\}.
\]
This moment $\mathcal N$ can be chosen greater than any other random moment $\mathcal M$.
To specify $\mathcal M$ in this part of the proof, we assume $x_0\in (K_{2\tau}, \, K_{2\tau+1})$ and define
$$ \mathcal M_1:=\inf \left\{i \in \mathbb
N: x_i\in \left( K_{2\theta+1}, \, \min\{\tilde G(\beta_0), K_{2\theta+2}\} \right) \right\},$$
 \begin{equation}
 \label{def:M12}
  \mathcal M_2:=\inf\{i> \mathcal M_1:x_i\in (K_{2\tau}, \, K_{2\tau+1})\},
 \end{equation}
 and, inductively, for $s\in \mathbb N$,
 \begin{equation*}
 \begin{split}
 & \mathcal M_{2s+1}:=\inf\{i>\mathcal M_{2s}:x_i\in (K_{2\theta+1}, \, \min\{\tilde G(\beta_0), K_{2\theta+2}\})\}, \\
 & \mathcal M_{2s+2}:=\inf\{i> \mathcal M_{2s+1}:x_i\in (K_{2\tau}, \, K_{2\tau+1})\}.
  \end{split}
  \end{equation*}
If $\mathbb P\left(\Omega_1\right)>0$, where $\Omega_1=\{\omega\in \Omega:  \mathcal M_i<\infty \,\, \mbox{for all} \,\, i\in \mathbb N \}$,  it means that,
with non-zero probability, a solution circulates infinitely many times between $(K_{2\tau}, \, K_{2\tau+1})$ and $(K_{2\theta+1}, \,\min\{\tilde
G(\beta_0), K_{2\theta+2}\})$. Recall that by  Remark \ref{rem:cases} circulation is possible only between those intervals. We are going to show that it is impossible for the control $\alpha$ and the noise level $\ell$ chosen as in \eqref{def:almr0}.

Let $\mathcal M:=\mathcal M_2$ be defined as in  \eqref{def:M12}, and $\mathcal  N>\mathcal M_2$ as described above.   Assuming that  $\mathbb P\left(\Omega_1\right)>0$, we get $x_n\in (K_{2\tau}+\sigma, \, K_{2\tau+1})\cup(K_{2\theta+1},\, \min\{\tilde G(\beta_0), K_{2\theta+2}\}-\sigma)$ for all $n\ge \mathcal  N\ge \mathcal M_2$, on $\Omega_1$.
The  solution can get  larger  than $K_{2\tau+1}-\delta $ in no more than $N_1$ steps.  If it gets into $(K_{2\tau+1}-\delta,  K_{2m+3})$, it stays there by definition of $\delta$ (since $K_{2\tau+1}-\delta<d_1(\alpha-\ell)$), \eqref{def:dc4eq} and since $\alpha_n\ge \alpha-\ell > \underline \beta_2$ for all $n\in \mathbb N$. This cannot happen with non-zero probability on $\Omega_1$, see also Remark~\ref{rem:cases}.
Since $\alpha_n\ge \alpha_*>\beta_3$, for any $n\in U_\mathcal N$, we have  $ \max_{x\in [K_{2\tau}, K_{2\tau+1}]}G(\alpha_*, x) <  \hat c(\alpha_*) $ (another case from \eqref{def:beta3mr} is treated similarly),  which implies
$ \max_{x\in [K_{2\tau}, K_{2\tau+1}]}G(\alpha_*, x) <  c_{k_1}(\alpha_*) $ for some $k_1\le k_0$. Since $G(\alpha_n, x) <G(\alpha_*, x)$ on $(K_{2\theta+1}, \min\{\tilde G(\beta_0), K_{2\theta+2}\})$, if $x_n$ gets over $K_{2\theta+1}$ it satisfies $x_n\le c_{k_1}(\alpha_*)$. Also, $x_n>K_{2\theta+1}+\delta$, since if $x_n\in (K_{2\theta+1}, K_{2\theta+1}+\delta)$ it stays in $(K_{2m+1}, K_{2\theta+1}+\delta)$, which, by definition of $\Omega_1$,  can happen only with the zero probability.  So the solution can get  below $K_{2\tau+1}$ in $s$ steps, $s\le N_2$ and
then it satisfies $x_{n+s}=G(\alpha_{n+s}, x_{n+s-1})\ge G(\alpha_*, x_{n+s-1}) \ge d_{k_1}(\alpha_*)$, where $x_{n+s-1}\in  (K_{2m+3}, c_{k_1}(\alpha_*))$.
So the solution will reach $(K_{2\tau+1}, K_{2\theta+1})$ in no more than $N$ steps. Since $\bar \beta_2>\alpha_n\ge \alpha-\ell\ge \underline \beta_2$ for all $n\in \mathbb N$, by Remark  \ref{rem:cases},  the interval $(K_{2\tau+1}, \, K_{2\theta+1})$ is a new trap, so the solution  stays there and converges to some equilibrium from $\mathbb K$.
The case when  $x_{\mathcal N}\in(K_{2\theta+1},\, \min\{\tilde G(\beta_0), K_{2\theta+2}-\sigma\}) $ is similar.

Assume now that $x_0\in (K_0, K_{2\tau})$.

If a solution gets into $(K_{2\tau}, \min\{\tilde G(\beta_0), K_{2\theta+2}\})$ on some $\Omega_2$ with $\mathbb P\left(\Omega_2\right)>0$, we denote $\mathcal M_{01}=\{i: x_i\in (K_{2m}, \min\{\tilde G(\beta_0), K_{2\theta+2}) \} ) \}$, consider $x_n$ with $n\ge\mathcal M_{01}$
and apply the above argument. The case when a solution remains in $(K_0, K_{2m})$ is covered by
Lemma~\ref{lem:detmr}.
The case $x_0\in (\min\{\tilde G(L), K_{2\theta+2}\},  K_{2(m+r)})$ is similar.

All the above implies  that the  infinite circulation can only happen with the zero probability, which concludes the proof.
\end{proof}

\medskip
Received May 2020; revised July 2021; early access November 2021.
\medskip

\end{document}